\documentclass[12pt]{amsart}

\usepackage{etex}
\usepackage{amsmath, amssymb}
\usepackage{array}
\usepackage[frame,cmtip,arrow,matrix,line,graph,curve]{xy}
\usepackage{graphpap, color, paralist, pstricks}
\usepackage[mathscr]{eucal}
\usepackage[pdftex]{graphicx}
\usepackage[pdftex,colorlinks,backref=page,citecolor=blue]{hyperref}
\usepackage{tikz}

\newtheorem{theorem}{Theorem}[section]
\newtheorem{proposition}[theorem]{Proposition}
\newtheorem{corollary}[theorem]{Corollary}
\newtheorem{lemma}[theorem]{Lemma}

\newtheorem{conjecture}[theorem]{Conjecture}

\theoremstyle{definition}
\newtheorem{definition}[theorem]{Definition}

\newtheorem{question}[theorem]{Question}
\newtheorem{remark}[theorem]{Remark}

\newcommand{\CC}{\mathbb{C} }
\newcommand{\PP}{\mathbb{P} }
\newcommand{\QQ}{\mathbb{Q} }
\newcommand{\RR}{\mathbb{R} }
\newcommand{\VV}{\mathbb{V} }
\newcommand{\ZZ}{\mathbb{Z} }

\newcommand{\cE}{\mathcal{E} }
\newcommand{\cF}{\mathcal{F} }

\newcommand{\cL}{\mathcal{L} }
\newcommand{\cM}{\mathcal{M} }

\newcommand{\cO}{\mathcal{O} }
\newcommand{\cQ}{\mathcal{Q} }

\newcommand{\cW}{\mathcal{W} }

\newcommand{\cY}{\mathcal{Y} }

\newcommand{\rH}{\mathrm{H} }

\newcommand{\rN}{\mathrm{N} }

\newcommand{\ba}{\mathbf{a} }
\newcommand{\bfb}{\mathbf{b} }

\newcommand{\bp}{\mathbf{p} }
\newcommand{\bq}{\mathbf{q} }

\newcommand{\bM}{\mathbf{M} }

\newcommand{\bQ}{\mathbf{Q} }
\newcommand{\bsq}{\mathbf{q} }
\newcommand{\bR}{\mathbf{R} }
\newcommand{\bS}{\mathbf{S} }

\def\Hom{\mathrm{Hom} }

\def\GL{\mathrm{GL}}
\def\SL{\mathrm{SL}}

\def\PGL{\mathrm{PGL}}

\def\Fl{\mathrm{Fl}}

\def\Pic{\mathrm{Pic} }

\def\rk{\mathrm{rank}\, }
\def\spec{\mathrm{Spec}\;}
\def\Cox{\mathrm{Cox}}

\def\proj{\mathrm{Proj}\;}
\def\Eff{\mathrm{Eff}}

\hypersetup{%
pdftitle={Finite generation of the algebra of type A conformal blocks II: higher genus},%
pdfauthor={Han-Bom Moon, Sang-Bum Yoo},%
pdfkeywords={moduli space, parabolic bundle, conformal block, Mori's program, Mori dream space},%
citecolor=blue,%
linkcolor=blue,%
}

\begin{document}

\title[Finite generation of the algebra of conformal blocks II]{Finite generation of the algebra of type A conformal blocks via birational geometry II: higher genus}
\date{\today}

\author{Han-Bom Moon}
\address{Department of Mathematics, Fordham University, New York, NY 10023, USA}
\email{hmoon8@fordham.edu}

\author{Sang-Bum Yoo}
\address{School of Natural Science, UNIST, 50 UNIST-gil, Ulsan 44919, Republic of Korea}
\email{sangbum.yoo@gmail.com}

\begin{abstract}
We prove finite generation of the algebra of type A conformal blocks over arbitrary stable curves of any genus. As an application we construct a flat family of irreducible normal projective varieties over the moduli stack of stable pointed curves, whose fiber over a smooth curve is a moduli space of semistable parabolic bundles. This generalizes a construction of a degeneration of the moduli space of vector bundles presented in a recent work of Belkale and Gibney.
\end{abstract}

\maketitle

\section{Introduction}

A \emph{Conformal block} is a representation theoretic object constructed as an example of two-dimensional chiral conformal field theory (\cite{TK88}, \cite{TUY89}, \cite{Uen08}). For each collection of data consisting of $X = (C, \bp) \in \overline{\cM}_{g, n}$, a simple Lie algebra $\mathfrak{g}$, $\ell \in \ZZ_{\ge 0}$, and a collection of dominant integral weights $\vec{\lambda} = (\lambda^{1}, \lambda^{2}, \cdots, \lambda^{n})$ such that $(\lambda^{i}, \theta) \le \ell$, there is a systematic way to construct a finite dimensional vector space $\VV_{X, \mathfrak{g}, \ell, \vec{\lambda}}^{\dagger}$, the space of conformal blocks. They form a locally free sheaf $\VV_{\mathfrak{g}, \ell, \vec{\lambda}}^{\dagger}$ over $\overline{\cM}_{g, n}$ and satisfy several functorial properties (Theorem \ref{thm:propagation}, Theorem \ref{thm:factorization}) motivated by conformal field theory.

Furthermore, conformal blocks have a surprising connection with algebraic geometry. For any $X \in \overline{\cM}_{g, n}$, $\VV_{X, \mathfrak{g}, \ell, \vec{\lambda}}^{\dagger}$ is realized as $\rH^{0}(\cM_{X}(G), L^{\ell, \vec{\lambda}})$, the space of global sections of a line bundle $L^{\ell, \vec{\lambda}}$ over the algebraic stack of parabolic principal $G$-bundles (\cite{BL94}, \cite{BF19}, \cite{Fal94}, \cite{KNR94}, \cite{LS97}, \cite{Pau96}). Since $\cM_{X}(G)$ is not proper, $h^{0}(\cM_{X}(G), L^{\ell, \vec{\lambda}}) < \infty$ is already non-trivial.

The identification $\VV_{X, \mathfrak{g}, \ell, \vec{\lambda}}^{\dagger} \cong \rH^{0}(\cM_{X}(G), L^{\ell, \vec{\lambda}})$ gives a multiplication map $\VV_{X, \mathfrak{g}, \ell, \vec{\lambda}}^{\dagger} \otimes \VV_{X, \mathfrak{g}, m, \vec{\mu}}^{\dagger} \to \VV_{X, \mathfrak{g}, \ell+m, \vec{\lambda}+\vec{\mu}}^{\dagger}$. Thus the sum of all conformal blocks
\[
	\VV_{X, \mathfrak{g}}^{\dagger} :=
	\bigoplus_{\ell, \vec{\lambda}}
	\VV_{X, \mathfrak{g}, \ell, \vec{\lambda}}^{\dagger}
\]
has a $\Pic(\cM_{X}(G))$-graded commutative $\CC$-algebra structure. This paper is mostly concerned with type A case ($\mathfrak{g} = \mathfrak{sl}_{r}$), so we set $\VV_{X}^{\dagger} = \VV_{X, \mathfrak{sl}_{r}}^{\dagger}$.

In this paper, we answer the question concerning the finite generation of $\VV_{X}^{\dagger}$. In \cite{MY17}, we prove the result for $X = (\PP^{1}, \bp) \in \cM_{0, n}$. When $n = 0$, i.e., when there is no marked point, Belkale and Gibney prove the finite generation of $\VV_{X}^{\dagger}$ in \cite{BG19}. Here we generalize these two results to arbitrary stable pointed curves. Our proof is completely different from \cite{BG19}.

\begin{theorem}[(\protect{Theorem \ref{thm:finitegenerationsingular}})]\label{thm:mainthmintro}
Let $X \in \overline{\cM}_{g, n}$ be any stable pointed curve. Then the algebra $\VV_{X}^{\dagger}$ of type A conformal blocks over $X$ is finitely generated.
\end{theorem}

\subsection{Degeneration of moduli spaces}

In this section we present an important application of Theorem \ref{thm:mainthmintro}.

Note that $\cM_{X}(\SL_{r})$ is naturally identified with $\cM_{X}(r, \cO)$, the moduli stack of rank $r$ (parabolic) vector bundles with trivial determinant. The geometry of the moduli space of semistable parabolic vector bundles has attracted many geometers throughout the last several decades. Since a very useful method to study geometry of an algebraic variety is to construct a flat degeneration with good geometric properties, one may wonder what a good degeneration of the moduli space of (parabolic) vector bundles is while we degenerate the base curve $X$ to a singular curve. As we have a natural smooth family $\cY \to \cM_{g, n}$, whose fiber over $X$ is the coarse moduli space $\bM_{X}(r, \cO, \ba)$ of parabolic vector bundles, one may obtain the degeneration by constructing a compactification $\overline{\cY}$ of $\cY$ with a morphism $\overline{\cY} \to \overline{\cM}_{g, n}$. For $n = 0$, there have been a number of results along this direction. See \cite[Section 11]{BG19} for the history and references.

Recently, in \cite[Theorem 1.2]{BG19}, Belkale and Gibney constructed a flat family $\cY \to \overline{\cM}_{g}$ of irreducible normal projective varieties whose fiber over a smooth curve is the moduli space of semistable vector bundles. As an application of Theorem \ref{thm:mainthmintro}, we obtain a generalization of their theorem to the moduli space of parabolic bundles.

\begin{theorem}\label{thm:applicationintro}
Fix a parabolic weight $\ba$. There is a flat family $\cY \to \overline{\cM}_{g, n}$ such that
\begin{enumerate}
\item For $X = (C, \bp) \in \cM_{g, n}$, $\cY_{X} \cong \bM_{X}(r, \cO, \ba)$;
\item $\cY_{X}$ is an irreducible normal projective variety for any $X \in \overline{\cM}_{g, n}$.
\end{enumerate}
\end{theorem}

When $n = 0$ (so $\bp = \emptyset$), our flat family is identical to that in \cite{BG19}.

Many compact moduli spaces are constructed as GIT quotients. Thus it depends on a choice of a polarization, which is in many cases not canonical. On the contrary, the flat family in Theorem \ref{thm:applicationintro} does not depend on any auxiliary numerical data other than the parabolic weight. Thus we may regard the flat family in Theorem \ref{thm:applicationintro} as a canonical one.

\subsection{Method of proof}

Conceptually, the procedures of proofs of Theorem \ref{thm:mainthmintro} and Theorem \ref{thm:applicationintro} are simple and straightforward. By \cite{Pau96}, for a smooth pointed curve $X \in \cM_{g, n}$, $\VV_{X}^{\dagger}$ is identified with the \emph{Cox ring} $\Cox(\cM_{X}(r, \cO))$ of the moduli stack of rank $r$ parabolic bundles with trivial determinant. As a first step, in Section \ref{ssec:codimestimation} and \ref{ssec:coxring} we find a parabolic weight $\ba$ such that $\Cox(\cM_{X}(r, \cO)) = \Cox(\bM_{X}(r, \cO, \ba))$ (under the assumption that the number of parabolic points is sufficiently large when $g = 1$). The crucial step (Section \ref{ssec:codimestimation}) is a careful computation of the dimension of the unstable locus in $\cM_{X}(r, \cO)$, following the method of Sun (\cite{Sun00}).

By computing the canonical divisor of $\bM_{X}(r, \cO, \ba)$ in Section \ref{ssec:canonicalbundle}, we prove that $\bM_{X}(r, \cO, \ba)$ is of Fano type. Then by the celebrated result of \cite{BCHM10}, we immediately obtain the finite generation of $\VV_{X}^{\dagger}$ in Section \ref{ssec:finitegeneration}.

We need to extend the result in two directions: 1) when the number of parabolic points is small in the $g = 1$ case and 2) when $X \in \overline{\cM}_{g, n}$ is singular. We employ two functorial properties of conformal blocks, the propagation of vacua (Theorem \ref{thm:propagation}) and the factorization (Theorem \ref{thm:factorization}) to describe $\VV_{X}^{\dagger}$ as a torus-invariant subring of a finitely generated algebra. Then we obtain the finite generation by using Nagata's theorem (\cite[Theorem 3.3]{Dol03}).

The proof of Theorem \ref{thm:applicationintro} is similar. We know that the sheaf of algebras of conformal blocks, defined over $\overline{\cM}_{g, n}$, is fiberwisely finitely generated. At the end of this paper, we describe a sheaf of torus invariant subalgebras over $\overline{\cM}_{g, n}$, whose relative $\proj$ is the flat family $\cY$ in Theorem \ref{thm:applicationintro}.

\subsection{Mori's program}

For a $\QQ$-factorial projective variety $Y$ with trivial irregularity, \emph{Mori's program} is a classification of all birational models of $Y$, which are equivalent or simpler than $Y$ (so-called rational contractions). More precisely, a rational contraction of $Y$ is a normal projective variety $Z$ with a rational map $Y \dashrightarrow Z$, which is a composition of flips, blow-downs, and fibrations, but not blow-ups.

Mori's program consists of three steps:
\begin{enumerate}
\item Study the cone $\Eff(Y) \subset \rN^{1}(Y)_{\RR}$ of effective divisors;
\item For each $D \in \Eff(Y)$, calculate the associated projective model
\begin{eqnarray}\label{eqn:projectivemodel}
	Y(D) := \proj \bigoplus_{m \ge 0}\rH^{0}(Y, \cO(mD));
\end{eqnarray}
\item Describe the associated rational contraction $Y \dashrightarrow Y(D)$.
\end{enumerate}

There are several theoretical obstructions to the completion of Mori's program including the finite generation of the section ring in \eqref{eqn:projectivemodel}. A \emph{Mori dream space} (MDS for short) is a projective variety without such difficulties (\cite{HK00}). Within the proof of Theorem \ref{thm:mainthmintro}, we show that for any general parabolic weight $\ba$, the coarse moduli space $\bM_{X}(r, \cO, \ba)$ is an MDS (Corollary \ref{cor:MDS}). So one may try to complete Mori's program for $\bM_{X}(r, \cO, \ba)$. In Section \ref{sec:Moriprogram}, we run Mori's program for $\bM_{X}(r, \cO, \ba)$.

\subsection{Related works}

Many questions, concerning the finite generation of $\VV_{X, \mathfrak{g}}^{\dagger}$ for other types, finding the set of effective generators, good presentations, cohomological properties such as being Gorenstein remain open except for a few cases. Here we list a few results from literature.

When $\mathfrak{g} = \mathfrak{sl}_{2}$, an explicit finite set of generators is described, for $X = (\PP^{1}, \bp) \in \cM_{0, n}$ by Castravet and Tevelev in \cite{CT06}, for a general smooth curve $X = C \in \cM_{g}$ by Abe in \cite{Abe10}, and for a general $X = (C, \bp) \in \cM_{g, n}$ by Manon in \cite{Man18}. For $\mathfrak{sl}_{3}$ and a general smooth curve $X = (C, \bp)$ with genus $\le 1$, Manon describes a similar explicit generating set in \cite{Man13}. Abe and Manon both use a degeneration of $\VV_{X}^{\dagger}$ to a singular curve.

Finally, the forementioned results of Manon also prove that $\VV_{X, \mathfrak{sl}_{r}}^{\dagger}$ is Gorenstein when $r \le 3$ and $X = (C, \bp)$ is a general pointed smooth curve.

\subsection{Some questions}

It is a very interesting problem to describe the degeneration of $\bM_{X}(r, \cO, \ba)$ in Theorem \ref{thm:applicationintro} as a moduli space of natural objects such as generalized parabolic sheaves on a singular curve.

\begin{question}
For $X \in \overline{\cM}_{g, n}\setminus \cM_{g, n}$, construct a projective moduli space of natural objects isomorphic to $\cY_{X}$ in Theorem \ref{thm:applicationintro}.
\end{question}

When $n = 0$, there has been an attempt to describe the flat limit as a moduli space of limit semistable bundles (\cite{Oss16}). See \cite[Section 11]{BG19} for a discussion.

We expect that the main result of this paper and the outline of the proof can be generalized to the other type of simple Lie algebras. We leave this to interested readers.

\begin{conjecture}
Let $\mathfrak{g}$ be a simple Lie algebra. For any stable curve $X \in \overline{\cM}_{g, n}$, the algebra $\VV_{X, \mathfrak{g}}^{\dagger}$ of conformal blocks is finitely generated.
\end{conjecture}

\subsection{Notation and conventions}

We work on an algebraically closed field $\CC$ of characteristic zero. To minimize the introduction of cumbersome notation, we discuss parabolic bundles with full flags only. An interested reader may generalize most parts of the paper to the partial flag cases. In many papers the dual $\VV_{X}$ of $\VV_{X}^{\dagger}$ has been denoted by the space of conformal blocks. All moduli stacks are defined over fppf topology.

\subsection*{Acknowledgement}

We would like to express our gratitude to Prakash Belkale and the anonymous referees for many suggestions. The first author would also like to thank the Institute for Advanced Study, where most of this work was done, for its hospitality and excellent research environment. The first author was partially supported by the Minerva Research Foundation.

\section{Moduli space of parabolic bundles and conformal blocks}

The purpose of this section is an introduction of notation and some well known results.

\subsection{Moduli stacks of parabolic vector bundles}

Let $C$ be a connected reduced projective curve of arithmetic genus $g$. Let $\bp = (p^{i})_{1 \le i \le n}$ be a collection of $n$ distinct smooth points on $C$. In this paper, $X = (C, \bp)$ denotes a pointed curve. We focus on the stable pointed curves, so $X = (C, \bp) \in \overline{\cM}_{g, n}$. We allow the non-pointed case, that is, $n = 0$ (and $\bp = \emptyset$).

\begin{definition}\label{def:parbundle}
Let $X = (C, \bp) \in \overline{\cM}_{g, n}$. A \emph{parabolic bundle} over $X$ of rank $r$ is a collection of data $\cE = (E, \{W_{\bullet}^{i}\})$ where
\begin{enumerate}
\item $E$ is a rank $r$ vector bundle on $C$;
\item For each $1 \le i \le n$, $W_{\bullet}^{i}$ is a strictly increasing sequence $0 \subsetneq W_{1}^{i} \subsetneq W_{2}^{i} \subsetneq \cdots \subsetneq W_{r-1}^{i} \subsetneq W_{r}^{i} = E|_{p^{i}}$ of subspaces of $E|_{p^{i}}$. In other words, $W_{\bullet}^{i} \in \Fl(E|_{p^{i}})$. Note that $\dim W_{j}^{i} = j$.
\end{enumerate}
\end{definition}

\begin{definition}\label{def:stackofparbundles}
Let $\underline{\cM}_{X}(r, d)$ be the moduli stack of parabolic bundles over $X$, whose underlying vector bundle is of rank $r$ and degree $d$.
\end{definition}

The moduli stack $\underline{\cM}_{X}(r, d)$ is a non-separated algebraic stack (not of finite type). When $n = 0$, so if $X = C$, $\underline{\cM}_{C}(r, d)$ is smooth (\cite[Section 6]{Wan11}). For $n > 0$ and $X = (C, \bp)$, there is a natural forgetful map
\begin{eqnarray*}
	f : \underline{\cM}_{X}(r, d) &\to& \underline{\cM}_{C}(r, d)\\
	\cE = (E, \{W_{\bullet}^{i}\}) & \mapsto & E
\end{eqnarray*}
and $f$ is a smooth morphism because each fiber of $f$ is the product $\Fl(V)^{n}$ of flag varieties for an $r$-dimensional vector space $V$. Thus $\underline{\cM}_{X}(r, d)$ is also smooth.

There is another functorial morphism
\begin{eqnarray*}
	\det : \underline{\cM}_{X}(r, d) &\to& \Pic^{d}(C)\\
	\cE = (E, \{W_{\bullet}^{i}\}) & \mapsto & \det E.
\end{eqnarray*}

\begin{definition}
For $L \in \Pic^{d}(C)$, let $\underline{\cM}_{X}(r, L)$ be the fiber $\det^{-1}(L)$. In other words, $\underline{\cM}_{X}(r, L)$ is the moduli stack of parabolic bundles with a fixed determinant $L$.
\end{definition}

There is an open substack of $\underline{\cM}_{X}(r, L)$ with a $\CC^{*}$-gerbe structure because any object $(E, \{W_{\bullet}^{i}\})$ has at least one dimensional automorphism group given by dilations. To reduce this gerbe structure, we introduce a rigidified stack.

\begin{definition}
Let $\cM_{X}(r, L)$ be the moduli stack of data $(E, \{W_{\bullet}^{i}\}, \phi)$ where $(E, \{W_{\bullet}^{i}\})$ is a rank $r$ parabolic bundle with determinant $L$ and $\phi : \det E \stackrel{\cong}{\to} L$.
\end{definition}

For any $L \in \Pic^{0}(C)$, we have $\underline{\cM}_{X}(r, L) \cong \underline{\cM}_{X}(r, \cO)$ because if we set $M = \sqrt[r]{L^{-1}} \in \Pic^{0}(C)$, we have an isomorphism
\begin{eqnarray*}
	\underline{\cM}_{X}(r, L) &\to& \underline{\cM}_{X}(r, \cO)\\
	\cE = (E, \{W_{\bullet}^{i}\}) &\mapsto & (E \otimes M, \{W_{\bullet}^{i}\}).
\end{eqnarray*}
A similar argument shows that $\cM_{X}(r, L) \cong \cM_{X}(r, \cO)$ for any $L \in \Pic^{0}(C)$.

\subsection{Coarse moduli spaces of parabolic bundles}

In this section, let $X = (C, \bp) \in \cM_{g, n}$ be a smooth pointed curve. Most of moduli stacks in the previous section are neither separated nor of finite type. To obtain a projective coarse moduli space, we need to impose a stability condition. In contrast to the case of ordinary vector bundles, where the slope-stability is the standard choice, there are potentially infinitely many different ways to define the stability, which depend on numerical data.

\begin{definition}\label{def:parweight}
Fix $n \in \ZZ_{\ge 0}$. A \emph{parabolic weight} is a collection of data $\ba = (a_{\bullet}^{1}, a_{\bullet}^{2}, \cdots, a_{\bullet}^{n})$ where each $a_{\bullet}^{i} = (1 > a_{1}^{i} > a_{2}^{i} > \cdots > a_{r-1}^{i} > a_{r}^{i} = 0)$ is a strictly decreasing sequence of non-negative rational numbers.
\end{definition}

\begin{remark}
We may assume that $a_{r}^{i} = 0$ in Definition \ref{def:parweight} by the normalization trick (\cite[Remark 2.5]{MY17}).
\end{remark}

Intuitively, for a parabolic bundle $(E, \{W_{\bullet}^{i}\})$, $a_{j}^{i}$ has the role of `weight' of $W_{j}^{i} \subset E|_{p^{i}}$.

The space $\cW^{0}$ of all parabolic weights is the interior of a closed polytope $\Delta_{r-1}^{n}$ where $\Delta_{r-1}$ is an $(r-1)$-dimensional simplex. Indeed, if we set $d_{j} := a_{j}^{i}-a_{j+1}^{i}$ and set $a_{0}^{i} = 1$, then the space of all sequences $a_{\bullet}^{i}$'s is $\{(d_{j})_{0 \le j \le r-1} \in \RR^{r} \;|\; 0 < d_{j} < 1, \sum_{j = 0}^{r-1}d_{j} = 1\}$, which is the interior of $\Delta_{r-1}$. Let $\cW := \Delta_{r-1}^{n}$ be the closure of $\cW^{0}$.

\begin{definition}	
Let $\cE := (E, \{W_{\bullet}^{i}\})$ be a parabolic bundle of rank $r$ and $\ba$ be a parabolic weight.
\begin{enumerate}
\item The \emph{parabolic degree} of $\cE$ with respect to $\ba$ is
\[
	\mathrm{pdeg}_{\ba} \cE := \deg E
	+ \sum_{i=1}^{n}\sum_{j=1}^{r}a_{j}^{i}.
\]
\item The \emph{parabolic slope} of $\cE$ with respect to $\ba$ is
\[
	\mu_{\ba}(\cE) := \frac{\mathrm{pdeg}_{\ba} \cE}{r}.
\]
\end{enumerate}
\end{definition}

Let $\cE := (E, \{W_{\bullet}^{i}\})$ be a parabolic bundle and let $\ba$ be a parabolic weight. Let $F \subset E$ be a subbundle. There is a natural induced flag structure $W|_{F \bullet}^{i}$ on $F|_{p^{i}}$ as follows. Let $\ell$ be the smallest index such that $\dim (W_{\ell}^{i} \cap F|_{p^{i}}) = j$. Then $W|_{F_{j}}^{i} := W_{\ell}^{i} \cap F|_{p^{i}}$. Furthermore, we may define the induced parabolic weight $\bfb = (b_{\bullet}^{i})$ where $b_{j}^{i} := a_{\ell}^{i}$. Thus we obtain a parabolic bundle $\cF = (F, \{W|_{F \bullet}^{i}\})$ with a parabolic weight $\bfb$. $\cF$ is called a parabolic subbundle.

\begin{definition}
A parabolic bundle $\cE = (E, \{W_{\bullet}^{i}\})$ is \emph{$\ba$-(semi-)stable} if for any parabolic bundle $\cF$, $\mu_{\bfb}(\cF) (\le) < \mu_{\ba}(\cE)$.
\end{definition}

\begin{definition}
Let $\ba$ be a parabolic weight. Let $\cM_{X}(r, L, \ba)$ be the open substack of $\cM_{X}(r, L)$ parametrizing $\ba$-semistable parabolic bundles.
\end{definition}

From now on, we exclusively work on the $L = \cO$ case.

The space $\cW^{0}$ of all parabolic weights have a finite chamber structure given by a finitely many hyperplanes. If there is a strictly semistable parabolic bundle $\cE = (E, \{W_{\bullet}^{i}\})$, then there is a unique maximal destabilizing subbundle $\cF = (F, \{W|_{F \bullet}^{i}\})$ such that $\mu_{\bfb}(\cF) = \mu_{\ba}(\cE)$. If $\mathrm{rank}\; F = s$ and $\deg F = d$, then there are subsets $J^{i} \subset [r]$ with $|J^{i}| = s$ such that
\[
	\frac{d + \sum_{i=1}^{n}\sum_{j \in J^{i}}a_{j}^{i}}{s} =
	\frac{\sum_{i=1}^{n}\sum_{j=1}^{r}a_{j}^{i}}{r},
\]
which is an affine hyperplane (denoted by $H(s, d, \{J^{i}\})$) on $\RR^{n(r-1)} = \{(a_{j}^{i})_{1 \le j \le r-1, 1 \le i \le n}\}$. If we pick a point $\ba$ on the complement of the union of such hyperplanes, then the stability coincides with the semistability. In this case, we say $\ba$ is \emph{general}. Since $\Delta_{r-1}^{n}$ is compact, there are only finitely many $H(s, d, \{J^{i}\})$'s which intersect $\Delta_{r-1}^{n}$. For a connected component of the complement, the (semi-)stability does not change. Thus, there are only finitely many essentially different stability conditions.

For $g \ge 1$, $\cM_{X}(r, \cO, \ba)$ is always nonempty (Corollary \ref{cor:nonempty}). However, when $g = 0$, for some weight $\ba$, the moduli stack $\cM_{X}(r, \cO, \ba)$ might be empty. We say $\ba$ is \emph{effective} if $\cM_{X}(r, \cO, \ba) \ne \emptyset$. The subset of effective weights in $\cW^{0}$ is described in \cite[Section 6.2]{MY17}.

Suppose that $\ba$ is general and effective. Then $\cM_{X}(r, \cO, \ba)$ is a proper Deligne-Mumford stack. Mehta and Seshadri construct its coarse moduli space in \cite{MS80}, which is smooth and projective.

\begin{definition}
Let $\bM_{X}(r, \cO, \ba)$ be the coarse moduli space of $\cM_{X}(r, \cO, \ba)$.
\end{definition}

\begin{remark}
When $\ba$ is not general (but effective), $\cM_{X}(r, \cO, \ba)$ has a good moduli space in the sense of \cite{Alp13}. We call this good moduli space as $\bM_{X}(r, \cO, \ba)$, too.
\end{remark}

In summary, we have the following diagram:
\begin{equation}\label{eqn:stackdiagram}
	\xymatrix{\cM_{X}(r, \cO, \ba) \ar@^{(->}[r]^{\iota} \ar[d]^{p}
	& \cM_{X}(r, \cO)\\
	\bM_{X}(r, \cO, \ba)}
\end{equation}
Here $\iota$ is an inclusion of the stack, $p$ is the structure morphism to the coarse moduli space.

\subsection{Wall-crossing}\label{ssec:wallcrossing}

The wall-crossing is a change of the moduli space $\bM_{X}(r, \cO, \ba)$ that appears when $\ba$ varies across a wall $H(s, d, \{J^{i}\})$.

Consider a general point $\ba \in H(s, d, \{J^{i}\})$ and a small open neighborhood of $\ba$ divided into two pieces by the wall. Let $H(s, d, \{J^{i}\})^{+}$ and $H(s, d, \{J^{i}\})^{-}$ be the two connected components such that
\[
	\frac{d + \sum_{i=1}^{n}\sum_{j \in J^{i}}a_{j}^{i}}{s} >
	\frac{\sum_{i=1}^{n}\sum_{j=1}^{r}a_{j}^{i}}{r}
	\quad
	\mbox{and}
	\quad
	\frac{d + \sum_{i=1}^{n}\sum_{j \in J^{i}}a_{j}^{i}}{s} <
	\frac{\sum_{i=1}^{n}\sum_{j=1}^{r}a_{j}^{i}}{r},
\]
respectively. Let $\ba^{+}$ (resp. $\ba^{-}$) be a point on $H(s, d, \{J^{i}\})^{+}$ (resp. $H(s, d, \{J^{i}\})^{-}$).

There are two functorial morphisms (\cite[Theorem 3.1]{BH95}, \cite[Section 7]{Tha96})
\[
	\xymatrix{\bM_{X}(r, \cO, \ba^{-}) \ar[rd]^{\phi^{-}} &&
	\bM_{X}(r, \cO, \ba^{+}) \ar[ld]_{\phi^{+}}\\
	& \bM_{X}(r, \cO, \ba).}
\]
Let $Y \subset \bM_{X}(r, \cO, \ba)$ be the locus such that one of $\phi^{\pm} : Y^{\pm} := {\phi^{\pm}}^{-1}(Y) \to Y$ is not an isomorphism. Then $\bM_{X}(r, \cO, \ba^{-}) \setminus Y^{-} \cong \bM_{X}(r, \cO, \ba) \setminus Y \cong \bM_{X}(r, \cO, \ba^{+}) \setminus Y^{+}$.

\begin{proposition}[(\protect{\cite[Section 7]{Tha96}})]\label{prop:wallcrossing}
The blow-up of $\bM_{X}(r, \cO, \ba^{-})$ along $Y^{-}$ is isomorphic to the blow-up of $\bM_{X}(r, \cO, \ba^{+})$ along $Y^{+}$.
\end{proposition}

\subsection{Line bundles on moduli stack of parabolic bundles}

Let $X = (C, \bp) \in \cM_{g, n}$ be a smooth pointed curve. There are two ways to construct the moduli stack $\cM_{X}(r, \cO)$. One way is to construct it as a limit of finite type quotient stacks and will be presented in Section \ref{sec:codimension}. Here we describe $\cM_{X}(r, \cO)$ as a double quotient stack. Pick $q \in C \setminus \bp$. The formal neighborhood of $q$ in $C$ can be identified with $\spec \CC((z))$. $C \setminus q$ is a smooth affine variety, so it is isomorphic to $\spec A_{C}$ for some finitely generated $\CC$-algebra $A_{C}$.

\begin{theorem}[(Uniformization theorem, \protect{\cite[Proposition 4.2]{Pau96}})]\label{thm:uniformization}
The moduli stack $\cM_{X}(r, \cO)$ is canonically isomorphic to the quotient stack
\[
	\SL_{r}(A_{C})\backslash \left(\SL_{r}(\CC((z)))/\SL_{r}(\CC[[z]])
	\times \Fl(\CC^{r})^{n}\right).
\]
\end{theorem}

See also \cite[Sections 1 and 3]{BL94} for the geometric intuition.

The Picard group $\mathrm{Pic}(\Fl(\CC^{r}))$ is isomorphic to $\ZZ^{r-1}$. By the Borel-Weil theorem, for any integer partition $\lambda = (\lambda_{1}\ge \lambda_{2} \ge \cdots \ge \lambda_{r-1}\ge 0)$, there is a unique line bundle $F_{\lambda} \in \mathrm{Pic}(\Fl(\CC^{r}))$ such that $\rH^{0}(\Fl(\CC^{r}), F_{\lambda})$ is an irreducible $\SL_{r}$-representation $V_{\lambda}$ associated to the partition $\lambda$. The pull-back of $F_{\lambda}$ by $\SL_{r}(\CC((z)))/\SL_{r}(\CC[[z]]) \times \Fl(\CC^{r})^{n} \to \Fl(\CC^{r})^{n} \stackrel{\pi_{i}}{\to} \Fl(\CC^{r})$ descends to $\cM_{X}(r, \cO)$ and gives a line bundle $F_{i, \lambda}$ on $\cM_{X}(r, \cO)$.

For any family of vector bundles $E$ parametrized by a $\CC$-scheme $S$, we may construct the \emph{determinant line bundle} $L_{S} := \det R^{1}\pi_{S *}E \otimes (\det \pi_{S *}E)^{-1}$, where $\pi_{S} : C \times S \to S$ is the projection. The collection of data $L_{S}$ forms a line bundle $\cL$ on $\cM_{X}(r, \cO)$. Indeed, $\cL$ and $F_{i, \lambda}$'s generate the Picard group of the moduli stack.

\begin{theorem}[(\protect{\cite[Theorem 1.1]{LS97}})]\label{thm:Picstack}
The Picard group $\mathrm{Pic}(\cM_{X}(r, \cO))$ is isomorphic to
\[
	\ZZ \cL \times \prod_{i=1}^{n}\Pic(\Fl(\CC^{r})) \cong \ZZ^{(r-1)n+1}.
\]
In particular, $\Pic(\cM_{X}(r, \cO))$ is freely generated by $\{\cL, F_{i, \omega_{j}}\}_{1 \le i \le n, 1 \le j \le r-1}$ where $\omega_{j}$ is the $j$-th fundamental weight.
\end{theorem}

\begin{remark}\label{rem:Picforsingularcurve}
When $X = (C, \bp)$ is a reducible singular curve, then Theorem \ref{thm:uniformization} and Theorem \ref{thm:Picstack} are no longer true as they are. The correct statements are proven by Belkale and Fakhruddin in \cite{BF19}. See \cite[Proposition 5.2]{BF19} for the generalization of Theorem \ref{thm:uniformization}. If $C$ has $m$ irreducible components, then $\mathrm{rank}\;\mathrm{Pic}(\cM_{X}(r, \cO)) = (r-1)n + m$ (\cite[Proposition 7.1]{BF19}). However, the construction of the determinant line bundle is still valid for any $X$.
\end{remark}

\subsection{Algebra of conformal blocks}\label{ssec:conformalblocks}

Fix a pointed stable curve $X = (C, \bp) \in \overline{\cM}_{g, n}$. Here we allow a singular curve. We fix a simple Lie algebra $\mathfrak{g}$, $\ell \in \ZZ_{\ge 0}$ and a collection of dominant integral weights $\vec{\lambda} := (\lambda^{1}, \lambda^{2}, \cdots, \lambda^{n})$ where $(\theta, \lambda^{i}) \le \ell$. Here $\theta$ is the highest root and $(-,-)$ is the normalized Killing form. By using representation theory of affine Lie algebra, Tsuchiya, Kaine, Ueno and Yamada constructed a finite dimensional vector space $\VV_{X, \mathfrak{g}, \ell, \vec{\lambda}}$ and its dual space $\VV_{X, \mathfrak{g}, \ell, \vec{\lambda}}^{\dagger}$, the so-called \emph{space of conformal blocks}.

Here we present a brief outline of the representation theoretic construction of conformal blocks. For the detail of construction, see \cite[Section 3.1]{Uen08}. For each dominant integral weight $\lambda$, there is an irreducible $\mathfrak{g}$-representation $V_{\lambda}$. If we fix $\ell \in \ZZ_{\ge 0}$ such that $(\theta, \lambda) \le \ell$, there is an integrable highest weight module $H_{\ell, \lambda} \supset V_{\lambda}$ of affine Kac-Moody algebra $\hat{\mathfrak{g}}$ of level $\ell$ associated to $\mathfrak{g}$. For a collection of $\ell \in \ZZ_{\ge 0}$ and dominant integral weights $\vec{\lambda} = (\lambda^{1}, \lambda^{2}, \cdots, \lambda^{n})$ with $(\theta, \lambda^{i}) \le \ell$, let $H_{\ell, \vec{\lambda}} := \bigotimes_{i=1}^{n}H_{\ell, \lambda^{i}}$. Now fix $X = (C, \bp) \in \overline{\cM}_{g, n}$. Let
\[
	\rH^{0}(C, \cO_{C}(*\sum_{i}p^{i})) := \varinjlim_{m}\rH^{0}(C, \cO_{C}(m\sum_{i}p^{i}))
\]
and $\mathfrak{\hat{g}}(X) := \mathfrak{g} \otimes \rH^{0}(C, \cO_{C}(*\sum_{i} p^{i}))$. Then it has a natural Lie algebra structure and there is a $\mathfrak{\hat{g}}(X)$-action on $H_{\ell, \vec{\lambda}}$. Now
\[
	\VV_{X, \mathfrak{g}, \ell,  \vec{\lambda}} := H_{\ell, \vec{\lambda}}/\mathfrak{\hat{g}}(X)H_{\ell, \vec{\lambda}}
\]
and
\[
\VV_{X, \mathfrak{g}, \ell, \vec{\lambda}}^{\dagger} := \Hom_{\CC}(\VV_{X, \mathfrak{g}, \ell, \vec{\lambda}}, \CC) = \Hom_{\mathfrak{\hat{g}}(X)}(H_{\ell, \vec{\lambda}}, \CC).
\]

In this paper, $\mathfrak{g} = \mathfrak{sl}_{r}$ and we use the notation $\VV_{X, \ell, \vec{\lambda}}^{\dagger}$ instead of $\VV_{X, \mathfrak{sl}_{r}, \ell, \vec{\lambda}}^{\dagger}$.

The construction of $\VV_{X, \ell, \vec{\lambda}}^{\dagger}$ can be relativized over any family of stable curves. They form  \emph{vector bundle of conformal blocks} $\VV_{\ell, \vec{\lambda}}^{\dagger}$ on the moduli stack $\overline{\cM}_{g, n}$ equipped with a projectively flat connection with logarithmic singularities along the boundary divisors.

The Chern characters of vector bundle of conformal blocks $\VV_{\ell, \vec{\lambda}}^{\dagger}$ defines a semisimple cohomological field theory. In particular, there are several functorial isomorphisms between vector bundle of conformal blocks. Among them, there are two prominent morphisms.

\begin{theorem}[(Propagation of vacua, \protect{\cite[Theorem 3.15]{Uen08}})]\label{thm:propagation}
Let $f : \overline{\cM}_{g, n+1} \to \overline{\cM}_{g, n}$ be the forgetful map of the last marked point. Then $f^{*}\VV_{\ell, (\lambda^{1}, \lambda^{2}, \cdots, \lambda^{n})}^{\dagger} \cong \VV_{\ell, (\lambda^{1}, \lambda^{2}, \cdots, \lambda^{n}, 0)}^{\dagger}$.
\end{theorem}

\begin{theorem}[(Factorization, \protect{\cite[Theorem 3.19]{Uen08}})]\label{thm:factorization}
\begin{enumerate}
\item Let $g : \overline{\cM}_{g-1, n+2} \to \overline{\cM}_{g, n}$ be the gluing map of the last two marked points. Then
\[
	g^{*}\VV_{\ell, (\lambda^{1}, \lambda^{2}, \cdots, \lambda^{n})}^{\dagger}
	\cong \bigoplus_{\mu, (\mu, \theta) \le \ell}
	\VV_{\ell, (\lambda^{1}, \lambda^{2}, \cdots, \lambda^{n}, \mu, \mu^{*})}^{\dagger}
\]
\item Let $I_{1} \sqcup I_{2} = [n]$ be a partition. Let $h : \overline{\cM}_{g_{1}, |I_{1}|+1} \times \overline{\cM}_{g_{2}, |I_{2}|+1} \to \overline{\cM}_{g, n}$ be the gluing map of last marked points. Then
\[
	h^{*}\VV_{\ell, (\lambda^{1}, \lambda^{2}, \cdots, \lambda^{n})}^{\dagger}
	\cong \bigoplus_{\mu, (\mu, \theta) \le \ell}
	\VV_{\ell, ((\lambda^{i})_{i \in I_{1}}, \mu)}^{\dagger} \otimes
	\VV_{\ell, ((\lambda^{i})_{i \in I_{2}}, \mu^{*})}^{\dagger}.
\]
\end{enumerate}
\end{theorem}

Fix a curve $X \in \overline{\cM}_{g, n}$. There is also a natural product morphism
\begin{equation}\label{eqn:tensormap}
	\VV_{X, \ell, \vec{\lambda}}^{\dagger} \otimes
	\VV_{X, m, \vec{\mu}}^{\dagger} \to
	\VV_{X, \ell+m, \vec{\lambda} + \vec{\mu}}^{\dagger}.
\end{equation}
From the representation theoretic viewpoint, this map can be constructed as follows. The tensor product $V_{\lambda}\otimes V_{\mu}$ of $\mathfrak{g}$-representations has a unique irreducible subrepresentation isomorphic to $V_{\lambda + \mu}$. Thus there is a canonical $\mathfrak{g}$-module morphism  $V_{\lambda + \mu} \to V_{\lambda} \otimes V_{\mu}$. This construction is extended to a morphism between representations of $\mathfrak{\hat{g}}$:
\[
	H_{\ell+m, \lambda + \mu} \to H_{\ell, \lambda} \otimes H_{m, \mu}.
\]
By taking tensor products, we obtain a map
\[
	H_{\ell+m, \vec{\lambda} + \vec{\mu}} \to H_{\ell, \vec{\lambda}} \otimes H_{m, \vec{\mu}}.
\]
Take the dual, then we have
\[
	H_{\ell, \vec{\lambda}}^{*} \otimes H_{m, \vec{\mu}}^{*} \to H_{\ell + m, \vec{\lambda} + \vec{\mu}}^{*}.
\]
One may check that $\mathfrak{\hat{g}}$-action is compatible and we may obtain the map in \eqref{eqn:tensormap} after taking quotients. For the detail of the construction, see \cite[Section 2.2]{Man18}. After the identification in Theorem \ref{thm:Paulystack}, we may obtain another description as the tensor product of sections. They are identical.

Thus the direct sum of all conformal blocks
\[
	\VV_{X}^{\dagger} := \bigoplus_{\ell, \vec{\lambda}}
	\VV_{X, \ell, \vec{\lambda}}^{\dagger}
\]
has a commutative graded $\CC$-algebra structure. This algebra is called the \emph{algebra of conformal blocks} of type A. We also obtain a sheaf of algebras $\displaystyle\VV^{\dagger} := \bigoplus_{\ell, \vec{\lambda}}\VV_{\ell, \vec{\lambda}}^{\dagger}$.

\begin{remark}
The factorization holds on the level of algebras of conformal blocks (\cite[Proposition 3.1]{Man18}). See Proposition \ref{prop:factorizationalgebra}.
\end{remark}

\subsection{Conformal blocks and moduli spaces of parabolic bundles}

A conformal block can be understood as a section of a line bundle on $\cM_{X}(r, \cO)$.

\begin{theorem}[(\protect{\cite[Propositions 6.5 and 6.6]{Pau96} when $X$ is smooth, \cite[Theorem 1.7]{BF19} for a singular $X$})]\label{thm:Paulystack}
Let $X \in \overline{\cM}_{g, n}$. There is a canonical isomorphism
\[
	\VV_{X, \ell, \vec{\lambda}}^{\dagger} \cong
	\rH^{0}(\cM_{X}(r, \cO), \cL^{\ell} \otimes
	\bigotimes_{i=1}^{n}F_{i, \lambda^{i}}).
\]
\end{theorem}

When $X \in \cM_{g, n}$, Pauly also proves a similar statement for the coarse moduli space.

\begin{theorem}\label{thm:Paulycoarse}
Fix $\ell \ge 0$ and let $\vec{\lambda} = (\lambda^{1}, \lambda^{2}, \cdots, \lambda^{n})$ be a collection of dominant integral weights such that $(\theta, \lambda^{i}) < \ell$. Let $\ba$ be the parabolic weight such that $a_{j}^{i} = \lambda_{j}^{i}/\ell$.
\begin{enumerate}
\item \protect{(\cite[Theorem 3.3]{Pau96})} Suppose that $r|\sum_{i=1}^{n}\sum_{j=1}^{r-1}\lambda_{j}^{i}$. The restriction of $\cL^{\ell}\otimes \bigotimes_{i=1}^{n}F_{i, \lambda^{i}}$ to $\cM_{X}(r, \cO, \ba)$ descends to $\bM_{X}(r, \cO, \ba)$. Moreover, it is an ample line bundle on $\bM_{X}(r, \cO, \ba)$.
\item \protect{(\cite[Proposition 5.2]{Pau96})} $\VV_{X, \ell, \vec{\lambda}}^{\dagger} \cong \rH^{0}(\bM_{X}(r, \cO, \ba), \cL^{\ell} \otimes \bigotimes_{i=1}^{n}F_{i, \lambda^{i}})$.
\end{enumerate}
\end{theorem}

\begin{remark}\label{rem:degeneratedpardata}
Note that if $\lambda^{i}$ is a dominant integral weight such that $\lambda_{j}^{i} = \lambda_{j+1}^{i}$ (including $\lambda_{r-1}^{i} = \lambda_{r}^{i} = 0$), then $a_{j}^{i} = a_{j+1}^{i}$, so $\ba$ is not a genuine parabolic weight in the sense of Definition \ref{def:parweight}. In this case, we may think of $\cM_{X}(r, \cO, \ba)$ and $\bM_{X}(r, \cO, \ba)$ as moduli of parabolic bundles with partial parabolic data. For details, see \cite[Section 2]{Pau96}.
\end{remark}

\begin{remark}
The results of \cite{Pau96} are written under the assumption that $g \ge 2$. However, for the standard construction of the moduli stack, the construction of the coarse moduli space by geometric invariant theory (GIT), and the identification of spaces of global sections, the assumption is not essential. Exceptions are:
\begin{enumerate}
\item When $g(C) = 0$, for some $\ba$, the moduli stack $\cM_{X}(r, \cO, \ba)$ may be empty. When $g = 1$, the moduli stack $\cM_{X}(r, \cO, \ba)$ may not have any stable object.
\item \cite[Proposition 5.2]{Pau96} depends on the codimension calculation of the complement of $\cM_{X}(r, \cO, \ba)$ in $\cM_{X}(r, \cO)$. This result may not be valid for $g \le 1$ and some $\ba$.
\end{enumerate}
\end{remark}

Let $\cM$ be a $\CC$-scheme or an algebraic $\CC$-stack whose Picard group is a finitely generated free abelian group. The \emph{Cox ring} of $\cM$ is
\[
	\Cox(\cM) := \bigoplus_{L \in \mathrm{Pic}(\cM)}
	\rH^{0}(\cM, L).
\]
It has a commutative $\mathrm{Pic}(\cM)$-graded $\CC$-algebra structure.

By combining Theorems \ref{thm:Picstack} and \ref{thm:Paulystack}, we obtain:

\begin{corollary}\label{cor:conformalblockCoxring}
Let $X \in \cM_{g, n}$. Then as $\CC$-algebras,
\[
	\VV_{X}^{\dagger} \cong \Cox(\cM_{X}(r, \cO)).
\]
\end{corollary}

\begin{remark}
When $X$ is singular, because $\mathrm{rank}\;\mathrm{Pic}(\cM_{X}(r, \cO))$ may jump (Remark \ref{rem:Picforsingularcurve}), Corollary \ref{cor:conformalblockCoxring} is not true in general. In this case $\VV_{X}^{\dagger}$ is a proper subalgebra of $\Cox(\cM_{X}(r, \cO))$.
\end{remark}

\section{Codimension estimation and a consequence}\label{sec:codimension}

To prove the main theorem for smooth curves with positive genus, we need to identify the Cox ring of the moduli stack $\cM_{X}(r, \cO)$ with that of the coarse moduli space $\bM_{X}(r, \cO, \ba)$ for some parabolic weight $\ba$. We adopt the idea of Sun in \cite[Section 5]{Sun00} with some refinement.

\subsection{A construction of the moduli space of parabolic bundles}

We fix a pointed smooth curve $X = (C, \bp = (p^{i}))$ of genus $g$ and an ample line bundle $\cO_{C}(1)$ of degree one on $C$. Let $m \in \ZZ_{\ge 0}$. Let $\bQ(m)$ be the Quot scheme parametrizing quotients $\cO(-m)^{\nu(m)} \to F \to 0$ of rank $r$ and degree zero coherent sheaves, where $\nu(m) = rm + r(1-g)$. Let $\Omega(m)$ be the locally closed subvariety of $\bQ(m)$, which parametrizes quotients $\cO(-m)^{\nu(m)} \to F \to 0$ such that $\rH^{1}(C, F(m)) = 0$, $\cO^{\nu(m)} \to F(m)$ induces an isomorphism $\rH^{0}(C, \cO^{\nu(m)}) \cong \rH^{0}(C, F(m))$, and $\wedge^{r}F \cong \cO$. Let $\cO_{\Omega(m) \times C}(-m)^{\nu(m)} \to \cF \to 0$ be the universal quotient on $\Omega(m) \times C$ and $\Fl(\cF|_{p^{i}})$ be the full-flag bundle at $p^{i}$. Let
\[
	\bR(m) := (\times_{\Omega(m)}\Fl(\cF|_{p^{i}}))_{1 \le i \le n}
\]
be the fiber product over $\Omega(m)$ and let $\pi_{m} : \bR(m) \to \Omega(m)$ be the projection. By abuse of notation, we denote $\pi_{m}^{*}\cF$ by $\cF$. Naturally, $\bR(m)$ parametrizes parabolic bundles and we have a universal family $(\cF, \{\cW_{\bullet}^{i}\})$.

Fix a general parabolic weight $\ba$. We denote the open subset parametrizing $\ba$-stable parabolic bundles by $\bR(m)^{s}(\ba) \subset \bR(m)$. For $m \gg 0$, $	\underline{\cM}_{X}(r, \cO, \ba) \cong [\bR(m)^{s}(\ba)/\GL_{\nu(m)}] \subset [\bR(m)/\GL_{\nu(m)}]$. Roughly, $[\bR(m)/\GL_{\nu(m)}]$ (resp. $[\bR(m)^{s}(\ba)/\GL_{\nu(m)}]$) may be thought of as a stack of parabolic bundles (resp. $\ba$-stable parabolic bundles) with bounded Castelnuovo-Mumford regularity.

One may also slightly modify this construction for the stack $\cM_{X}(r, \cO, \ba)$, too. Let $q : \bR(m) \times C \to \bR(m)$ be the first projection and let $\cO(-m)^{\nu(m)} \to \cF \to 0$ be the universal quotient. Consider the sheaf $q_{*}\wedge^{r}\cF$ on $\bR(m)$. Let $\bS(m)$ be the complement of the zero section of $q_{*}\wedge^{r}\cF$ over $\bR(m)$. Naturally, $\bS(m)$ has a $\CC^{*}$-torsor structure and parametrizes quotients $\cO(-m)^{\nu(m)} \to \cF \to 0$ with an isomorphism $\wedge^{r}\cF \cong \cO$. Let $\bS(m)^{s}(\ba)$ be the locus of $\ba$-stable parabolic bundles. Then for $m \gg 0$,
\[
	\cM_{X}(r, \cO, \ba) \cong [\bS(m)^{s}(\ba)/\GL_{\nu(m)}]\subset [\bS(m)/\GL_{\nu(m)}].
\]

There is a canonical open embedding $[\bS(m)/\GL_{\nu(m)}] \hookrightarrow [\bS(m+1)/\GL_{\nu(m+1)}]$, which maps $(F, \{W_{\bullet}^{i}\})$ to itself. Furthermore, any $(F, \{W_{\bullet}^{i}\}, \phi) \in \cM_{X}(r, \cO)$ is in $[\bS(m)/\GL_{\nu(m)}]$ for some $m$. Thus we obtain:

\begin{lemma}
\[
	\cM_{X}(r, \cO) \cong
    \varinjlim_{m}[\bS(m)/\GL_{\nu(m)}].
\]
\end{lemma}

\begin{corollary}\label{cor:globalsectionlimit}
Let $L$ be any line bundle on $\cM_{X}(r, \cO)$ and let $i_{m} : [\bS(m)/\GL_{\nu(m)}] \hookrightarrow \cM_{X}(r, \cO)$ be the inclusion. Then
\[
	\rH^{0}(\cM_{X}(r, \cO), L) \cong \varprojlim_{m}
	\rH^{0}([\bS(m)/\GL_{\nu(m)}], i_{m}^{*}L).
\]
\end{corollary}

\begin{proof}
By the universal property of the inverse limit, we have a morphism
\[
	\rH^{0}(\cM_{X}(r, \cO), L) \to
	\varprojlim_{m}\rH^{0}([\bS(m)/\GL_{\nu(m)}], i_{m}^{*}L).
\]
Since $[\bS(m)/\GL_{\nu(m)}] \to [\bS(m+1)/\GL_{\nu(m+1)}]$ is an open embedding for all $m$, the bijectivity is immediate.
\end{proof}

\begin{corollary}\label{cor:picardgplimit}
Under the same notation, there is an isomorphism
\[
	\Pic(\cM_{X}(r, \cO)) \cong \varprojlim_{m}
	\Pic([\bS(m)/\GL_{\nu(m)}]).
\]
\end{corollary}

\begin{proof}
As before, there is a functorial morphism
\[
	r : \Pic(\cM_{X}(r, \cO)) \to \varprojlim_{m}
	\Pic([\bS(m)/\GL_{\nu(m)}]).
\]
Since $\cM_{X}(r, \cO)$ is a smooth stack, each pull-back $i_{m}^{*} : \Pic(\cM_{X}(r, \cO)) \to \Pic([\bS(m)/\GL_{\nu(m)}])$ is surjective (\cite[Corollary 3.4]{Hei10}). The surjectivity of $r$ follows from Zorn's lemma. Suppose that $L \in \ker r$. Then for each $m$, $i_{m}^{*}L$ is trivial and both $i_{m}^{*}L$ and its dual have non-vanishing global sections. By Corollary \ref{cor:globalsectionlimit}, those sections are extended to non-vanishing global sections of $L$ and $L^{*}$ on $\cM_{X}(r, \cO)$. Thus $L$ is trivial and $r$ is injective.
\end{proof}

\subsection{Codimension estimation}\label{ssec:codimestimation}

The aim of this section is to prove the following result, based on the argument of Sun (\cite[Section 5]{Sun00}). As in the previous section, $X = (C, \bp)$ is a smooth pointed curve of genus $g$, and $\ba$ is a general parabolic weight.

\begin{definition}\label{def:weightac}
Let $\ba_{c} = (a_{c \bullet}^{i})$ be a parabolic weight where $a_{c \bullet}^{i} = \frac{1}{r}(r-1, r-2, \cdots, 1)$.
\end{definition}

\begin{proposition}\label{prop:codimestimation}
Let $\bS(m)^{us}(\ba) := \bS(m) \setminus \bS(m)^{s}(\ba)$. Suppose that $m \gg 0$. Let $\mathrm{codim}\;\bS(m)^{us}(\ba)$ be the codimension of $\bS(m)^{us}(\ba)$ in $\bS(m)$. Recall that $g$ is the genus of $C$, $r$ is the rank of the parabolic bundles, and $n$ is the number of parabolic points. Then:
\begin{enumerate}
\item $\mathrm{codim}\; \bS(m)^{us}(\ba) \ge (r-1)(g-1)+1$. In particular, if $g \ge 2$, $\mathrm{codim}\; \bS(m)^{us}(\ba) \ge 2$.
\item Suppose that $g = 1$ and $n > r$. Let $\ba$ be a general weight very close to $\ba_{c}$. Then $\mathrm{codim}\; \bS(m)^{us}(\ba) \ge 2$.
\end{enumerate}
\end{proposition}

\begin{proof}
Because the proof is similar to that of \cite[Proposition 5.1]{Sun00}, here we only give a brief outline. For details, see \cite{Sun00}.

It is clear that the codimension of $\bR(m)^{us}(\ba) := \bR(m)\setminus \bR(m)^{s}(\ba)$ in $\bR(m)$ is equal to that of $\bS(m)^{us}(\ba)$ in $\bS(m)$. So for notational simplicity, we calculate the codimension of $\bR(m)^{us}$ in $\bR(m)$. Note that $\dim \bR(m) = r^{2}(g-1) + \nu(m)^{2} + nr(r-1)/2$.

Let $\cE = (E, \{W_{\bullet}^{i}\}) \in \bR(m)^{us}(\ba)$. Then there is a unique maximal destabilizing subbundle $\cE_{1} = (E_{1}, \{W|_{E_{1} \bullet}^{i}\})$ of rank $r_{1}$ and degree $d_{1}$. $E$ fits into the exact sequence
\[
	0 \to E_{1} \to E \to E_{2} \to 0.
\]
Let $r_{2}$ (resp. $d_{2}$) be the rank (resp. degree) of $E_{2}$, so $r = r_{1}+ r_{2}$ and $d_{1} + d_{2} = 0$. There is a collection $\{J^{i}\}$ of subsets $J^{i} \subset [n]$ with $|J^{i}| = r_{1}$ such that
\[
	\mu_{\bfb}(\cE_{1}) = \frac{d +
	\sum_{i=1}^{n}\sum_{j \in J^{i}}a_{j}^{i}}{r_{1}}>\frac{\sum_{i=1}^{n}\sum_{j=1}^{r} a_{j}^{i}}{r}=\mu_{\ba}(\cE).
\]

We construct a variety $F(Y)$ parametrizing such $\cE$ as the following. Let $\bQ_{k}$ ($k = 1, 2$) be the Quot scheme that parametrizes rank $r_{k}$, degree $d_{k}$ quotients $\cO(-m_{k})^{\nu_{k}} \to E_{k} \to 0$. Here $\nu_{k} = r_{k}m_{k} + r_{k}(1-g)$, and $m_{k}$ is taken sufficiently large so that $\rH^{1}(C, E_{k}(m_{k})) = 0$ and $\rH^{0}(C, E_{k}(m_{k})) \cong \CC^{\nu_{k}}$. Let $\cF_{k}$ be the universal quotient on $\bQ_{k} \times C$. Note that there are morphisms $\bQ_{k} \to \Pic^{d_{k}}(C)$ which map $\cO(-m_{k})^{\nu_{k}} \to E_{k} \to 0$ to $(-1)^{k}\det E_{k}$. Let $B = \bQ_{1} \times_{\Pic^{d_{2}}(C)} \bQ_{2}$ and $\cF := \cF_{2}^{\vee} \otimes \cF_{1}$ on $B \times C$. For the projection $\pi : B \times C \to B$, set
\[
	B_{h} := \{x \in B \;|\; h^{1}(\pi^{-1}(x), \cF|_{\pi^{-1}(x)})
	= h\}.
\]
Then $B_{h}$ is a locally closed subscheme of $B$ and $B = \cup_{h \ge 0}B_{h}$. Note that over each $B_{h}$, $R^{1}\pi_{*}\cF$ is locally free of rank $h$.

We define $P_{h}$ as:
\begin{enumerate}
\item[(1)] $P_{0} = B$ and $\cF^{0} := \cF_{1}\oplus \cF_{2}$;
\item[(2)] for $h > 0$, $P_{h} := \PP (R^{1}\pi_{*}\cF)$ and $\cF^{h}$ is the universal extension $0 \to \cF_{1} \otimes \cO(1) \to \cF^{h} \to \cF_{2} \to 0$.
\end{enumerate}

We denote the $j$-th element of $J^{i}$ by $J_{j}^{i}$ and set $J_{0}^{i} = 0$. For each $i \in [n]$, let $Y^{i}$ be the subvariety of the flag bundle $\Fl(\cF^{h}|_{p^{i}})$ defined by
\[
	Y^{i} = \{W_{\bullet}^{i}\in \Fl(\cF^{h}|_{p^{i}})\;|\;
	\dim (W_{j}^{i} \cap E_{1}|_{p^{i}}) = d \mbox{ if and only if }
	J_{d}^{i} \le j < J_{d+1}^{i}\}.
\]
Since this is a fiber bundle over $P_{h}$ whose fiber is a Schubert variety, we may compute its codimension, which is
\begin{equation}\label{eqn:codimSchubert}
	\sum_{j=1}^{r_{1}}(r-r_{1}-J_{j}^{i}+j).
\end{equation}

Finally, let $Y = (\times_{P_{h}}Y^{i})_{1 \le i \le n}$ and let $F(Y)$ be the frame bundle of the pull-back of $\cF^{h}$ over $Y$. Then $F(Y)$ parametrizes triples consisting of a framed unstable parabolic bundle, a framed destabilizing bundle and a framed quotient bundle. By measuring the dimension of $F(Y)$ and that of a general fiber of $F(Y) \to \bR(m)$, one can compute the codimension of the image to be at least
\begin{equation}\label{eqn:codim}
	r_{1}r_{2}(g-1) + \sum_{i=1}^{n} \mathrm{codim}\; Y^{i} + rd_{1}
\end{equation}
(\cite[p.513]{Sun00}).

Any unstable parabolic bundle is in the image of $F(Y)$ with some $r_{1}$, $d_{1}$, $\{J^{i}\}$, and some $h$. Since there are only countably many choices of those numerical data, one can conclude that the codimension of $\bR(m)^{us}(\ba)$ is at least the minimum of \eqref{eqn:codim}.

Since $\cE_{1}$ is the maximal destabilizing subbundle, $\mu_{\bfb}(\cE_{1}) > \mu_{\ba}(\cE)$, or equivalently,
\begin{equation}\label{eqn:boundofrd}
	rd_{1} > \sum_{i=1}^{n}\left(\sum_{j =1}^{r-1}r_{1}a_{j}^{i}
	- \sum_{j \in J^{i}}ra_{j}^{i}\right).
\end{equation}
By combining \eqref{eqn:codimSchubert}, \eqref{eqn:codim}, and \eqref{eqn:boundofrd}, we can conclude that
\[
	\mathrm{codim}\; \bR(m)^{us}(\ba) > (r-1)(g-1) + \sum_{i=1}^{n}
	\left(\sum_{j=1}^{r_{1}}(r-r_{1}-J_{j}^{i}+j) +
	\sum_{j=1}^{r-1}r_{1}a_{j}^{i} - \sum_{j \in J^{i}}ra_{j}^{i}\right).
\]
A combinatorial computation (\cite[Lemma 5.2]{Sun00}) says that the sum in the big parenthesis on the right hand side is strictly positive. So Item (i) follows immediately. For Item (ii), note that each term on the right hand side is either an integer or a multiple of the same unit, which is a slight perturbation of $\frac{1}{r}$. Therefore the number inside the parenthesis is at least $\frac{1}{r} - \epsilon$ for some $\epsilon > 0$. If $n > r$, then the right hand side is larger than one, so we have the same conclusion.
\end{proof}

\begin{remark}\label{rmk:codimofsslocus}
By using the same idea, even in the case of non-general weight $\ba$, under the same assumption ($g \ge 2$ or $g = 1$, $n > r$, and $\ba$ being close to $\ba_{c}$), one may check that the codimension of the strictly semistable locus $\bS(m)^{ss}(\ba) \setminus \bS(m)^{s}(\ba)$ in $\bS(m)^{ss}(\ba)$ is at least two. The only difference is that the strict inequality in \eqref{eqn:boundofrd} must be replaced by the equality. The other steps of the computation are identical.
\end{remark}

\begin{definition}\label{def:dominantweight}
A general parabolic weight $\ba$ is \emph{dominant} if the codimension of $\bS(m)^{us}(\ba)$ in $\bS(m)$ is at least two for some $m \gg 0$.
\end{definition}

Proposition \ref{prop:codimestimation} provides the existence of dominant weights for all $g \ge 1$. Note that when $g \ge 2$, every general weight is dominant.

Another immediate consequence is the non-emptiness of the stack.

\begin{corollary}\label{cor:nonempty}
Suppose that $g \ge 1$. Then every parabolic weight is effective.
\end{corollary}

\begin{proof}
Item (i) of Proposition \ref{prop:codimestimation} shows that the codimension of the unstable locus is positive. This implies that there must be a semistable bundle. Thus $\cM_{X}(r, \cO, \ba)$ is nonempty.
\end{proof}

\section{Finite generation}\label{sec:finitegeneration}

In this section we prove Theorem \ref{thm:mainthmintro} for a smooth pointed curve $X = (C, \bp) \in \cM_{g, n}$ with $g \ge 1$. The genus zero case was shown in \cite{MY17}.

\subsection{Comparison of Picard groups}\label{ssec:Picardgroup}

\begin{proposition}\label{prop:Picardgroupstackandcoarsemoduli}
Let $\ba$ be a general parabolic weight. Then $\Pic(\bM_{X}(r, \cO, \ba))$ is an index $r$ sublattice of $\Pic(\cM_{X}(r, \cO, \ba))$.
\end{proposition}

\begin{proof}
Let $p : \cM_{X}(r, \cO, \ba) \to \bM_{X}(r, \cO, \ba)$ be the structure morphism. Note that $p_{*}\cO_{\cM_{X}(r, \cO, \ba)}\\ \cong \cO_{\bM_{X}(r, \cO, \ba)}$, since $\bM_{X}(r, \cO, \ba)$ is the coarse moduli space (\cite[Theorem 11.1.2]{Ols16}). So if $p^{*}L \cong \cO_{\cM_{X}(r, \cO, \ba)}$, then $L \cong p_{*}p^{*}L \cong \cO_{\bM_{X}(r, \cO, \ba)}$ by the projection formula. Thus $p^{*} : \Pic(\bM_{X}(r, \cO, \ba)) \to \Pic(\cM_{X}(r, \cO, \ba))$ is injective. Indeed, $\Pic(\bM_{X}(r, \cO, \ba))$ is a finite index subgroup of $\Pic(\cM_{X}(r, \cO, \ba))$ (\cite[Lemma 2]{KV04}). Furthermore, in the proof of \cite[Theorem 3.3]{Pau96}, Pauly shows that the line bundle $\cL^{\ell}\otimes \bigotimes_{i=1}^{n}F_{i, \lambda^{i}}$ descends to $\bM_{X}(r, \cO, \ba)$ if and only if $r|\sum_{i=1}^{n}\sum_{j=1}^{r-1}\lambda_{j}^{i}$. Thus $\mathrm{im}\; p^{*}$ is an index $r$ sublattice of $\Pic(\cM_{X}(r, \cO, \ba))$.
\end{proof}

\begin{proposition}\label{prop:Picardgroupforcoarsemoduli}
Let $\ba$ be a dominant parabolic weight. Then $\Pic(\cM_{X}(r, \cO, \ba)) \cong \Pic(\cM_{X}(r, \cO))\\ \cong \ZZ^{n(r-1)+1}$.
\end{proposition}

\begin{proof}
Assume that $m \gg 0$. Then the diagram in \eqref{eqn:stackdiagram} extends to
\[
	\xymatrix{\cM_{X}(r, \cO, \ba) =
	\left[\bS(m)^{s}(\ba)/\GL_{\nu(m)}\right]\ar@^{(->}[r]^{\quad\quad\quad\quad j}
	\ar[d]^{p}&
	\left[\bS(m)/\GL_{\nu(m)}\right] \ar@^{(->}[r]^{\quad\iota_{m}}&
	\cM_{X}(r, \cO)\\
	\bM_{X}(r, \cO, \ba).}
\]
Proposition \ref{prop:codimestimation} and \cite[Proposition 18]{EG98} imply that $j^{*}:\Pic(\left[\bS(m)/\GL_{\nu(m)}\right]) \to \Pic(\cM_{X}(r, \cO, \ba))$ is an isomorphism. In particular, the natural inclusion $[\bS(m)/\GL_{\nu(m)}] \subset [\bS(m+1)/\GL_{\nu(m+1)}]$ induces an isomorphism $\Pic(\left[\bS(m)/\GL_{\nu(m)}\right]) \cong \Pic(\left[\bS(m+1)/\GL_{\nu(m+1)}\right])$. Thus
\[
	\varprojlim_{\ell}\Pic(\left[\bS(\ell)/\GL_{\nu(\ell)}\right]) \cong \Pic(\left[\bS(m)/\GL_{\nu(m)}\right])
\]
for some $m \gg 0$. By Corollary \ref{cor:picardgplimit}, we obtain an isomorphism $\Pic(\cM_{X}(r, \cO, \ba)) \cong \Pic ([\bS(m)/\GL_{\nu(m)}])\\ \cong \Pic(\cM_{X}(r, \cO))$.

\end{proof}

\begin{remark}
In \cite{Pau96} Pauly proves that a single explicit line bundle constructed from $\ell$ and $\lambda^{i}$ descends to $\bM_{X}(r, \cO, \ba)$. Because his proof relies on a weight computation, the same proof is valid for any line bundle satisfying the divisibility condition.
\end{remark}

\subsection{Comparison of Cox rings}\label{ssec:coxring}

\begin{proposition}\label{prop:coxringidentification}
Let $\ba$ be a dominant parabolic weight. Then $\Cox(\cM_{X}(r, \cO)) \cong \Cox(\bM_{X}(r, \cO, \ba))$.
\end{proposition}

\begin{proof}
\textsf{Step 1.}
By Proposition \ref{prop:Picardgroupforcoarsemoduli}, we may identify $\Pic(\cM_{X}(r, \cO))$ with $\Pic(\cM_{X}(r, \cO, \ba))$. For any $L \in \Pic(\cM_{X}(r, \cO))$, by Corollary \ref{cor:globalsectionlimit} and the codimension estimation,
\[
	\rH^{0}(\cM_{X}(r, \cO, \ba), L)
	= \rH^{0}([\bS(m)/\GL_{\nu(m)}], L)
	= \varprojlim_{\ell}\rH^{0}([\bS(\ell)/\GL_{\nu(\ell)}], L)
	= \rH^{0}(\cM_{X}(r, \cO), L)
\]
for some $m \gg 0$. Thus we have $\Cox(\cM_{X}(r, \cO)) \cong \Cox(\cM_{X}(r, \cO, \ba))$.

\textsf{Step 2.}
Suppose that $L \in \mathrm{im}\; p^{*} \subset \Pic(\cM_{X}(r, \cO, \ba))$, so $L = p^{*}L'$. Then by the projection formula, $\rH^{0}(\cM_{X}(r, \cO, \ba), p^{*}L') = 	\rH^{0}(\bM_{X}(r, \cO, \ba), p_{*}p^{*}L') = \rH^{0}(\bM_{X}(r, \cO, \ba), L')$.

\textsf{Step 3.}
We show that if $L \notin \mathrm{im} \; p^{*}$, then $\rH^{0}(\cM_{X}(r, \cO), L) = 0$. Note that $\rH^{0}(\cM_{X}(r, \cO), L) = \VV^{\dagger}_{X, \ell, \vec{\lambda}}$ for some $\ell$ and $\vec{\lambda}$. When $L \notin \mathrm{im}\;p^{*}$, $|\vec{\lambda}| := \sum_{i=1}^{n}|\lambda^{i}| = \sum_{i=1}^{n}\sum_{j=1}^{r-1}\lambda_{j}^{i}$ is not a multiple of $r$.

Take a degeneration $X' = (C', \bp')$ of $X = (C, \bp)$ on $\overline{\cM}_{g, n}$ such that $C'$ is irreducible and has $g$ double points so its normalization $\widetilde{C}$ is a rational curve. Let $\widetilde{X} = (\widetilde{C}, \bq) \in \overline{\cM}_{0, n+2g}$ where $\bq$ consists of the inverse image of singular points and marked points $\bp'$ via the normalization map. Since conformal blocks form a vector bundle over $\overline{\cM}_{g, n}$, $\rk \VV^{\dagger}_{X, \ell, \vec{\lambda}} = \rk \VV^{\dagger}_{X', \ell, \vec{\lambda}}$. By the factorization (Theorem \ref{thm:factorization}),
\[
	\rk \VV^{\dagger}_{X', \ell, \vec{\lambda}}
	= \sum_{\vec{\mu}} \rk \VV^{\dagger}_{\widetilde{X}, \ell, \vec{\lambda} \cup \vec{\mu} \cup \vec{\mu^{*}}},
\]
where the sum is taken over all $g$-sequences $\vec{\mu} = (\mu^{1}, \mu^{2}, \cdots, \mu^{g})$ of dominant integral weights with $(\mu^{k}, \theta) \le \ell$. Here $\vec{\mu^{*}} = (\mu^{1 *}, \mu^{2 *}, \cdots, \mu^{g *})$.

Since $|\mu^{k}| + |\mu^{k *}| = r\ell$, $|\vec{\lambda}| + |\vec{\mu}| + |\vec{\mu^{*}}|$ is not a multiple of $r$ either. But for $g = 0$, $\VV^{\dagger}_{\widetilde{X}, \ell, \vec{\lambda} \cup \vec{\mu} \cup \vec{\mu^{*}}}$ is a subspace of the space of invariants $\left(\bigotimes_{i} V_{\lambda^{i}} \otimes \left(\bigotimes_{k} V_{\mu^{k}} \otimes V_{\mu^{k *}}\right)\right)^{\SL_{r}}$, which is trivial.

\textsf{Step 4.}
In summary,
\[
\begin{split}
	\Cox(\cM_{X}(r, \cO)) &= \Cox(\cM_{X}(r, \cO, \ba)) =
	\bigoplus_{L \in \Pic(\cM_{X}(r, \cO, \ba))}
	\rH^{0}(\cM_{X}(r, \cO, \ba), L)\\
	&= \bigoplus_{L \in \Pic(\bM_{X}(r, \cO, \ba))}
	\rH^{0}(\cM_{X}(r, \cO, \ba), p^{*}L) \\
	&=
	\bigoplus_{L \in \Pic(\bM_{X}(r, \cO, \ba))}
	\rH^{0}(\bM_{X}(r, \cO, \ba), L) = \Cox(\bM_{X}(r, \cO, \ba)).
\end{split}
\]
\end{proof}

\subsection{Canonical bundle}\label{ssec:canonicalbundle}

Fix $X = (C, \bp) \in \cM_{g, n}$. Let $q:\bR(m)^{ss}(\ba)\to \bM_{X}(r,\cO,\ba)$ be the quotient map. Let $(\cE,\{\cW_{\bullet}^{i}\})$ be the universal family over $\bR(m)^{ss}(\ba)\times C$. Let $\cQ_{p^i,j} :=\cE|_{\bR(m)^{ss}\times\{p^i\}}/\cW_{j}^{i}$ and let $\pi: \bR(m)^{ss}(\ba) \times C \to \bR(m)^{ss}(\ba)$ be the projection.

\begin{proposition}[(\protect{\cite[Proposition 2.2 and (5.9)]{Sun00}})]
Fix $y \in C$. For any parabolic weight $\ba$,
\[
	\omega_{\bR(m)^{ss}(\ba)}^{-1} \cong
	(\det R\pi_{*}\cE)^{2r}\otimes
	\left(\bigotimes_{i=1}^{n}\bigotimes_{j=1}^{r-1}
	(\det\cQ_{p^i,j})^{2}\right)
	\otimes(\det\cE|_{\bR(m)^{ss}(\ba)\times\{y\}})^{2r(1-g)-n(r-1)}.
\]
\end{proposition}

\begin{remark}
Sun proves the statement for $\widetilde{\bR}(m)^{ss}(\ba)$ which is the parameter space such that $\widetilde{\bR}(m)^{ss}(\ba)/\PGL_{\nu(m)} \cong \bM_{X}(r, 0, \ba)$, instead of $\bM_{X}(r, \cO, \ba)$. However, the same formula also works for $\bR(m)^{ss}(\ba)$. Indeed, there is a determinant map
\begin{eqnarray*}
	\det : \widetilde{\bR}(m)^{ss}(\ba) & \to & \Pic^{0}(C)\\
	(\cO(-m)^{\nu(m)} \to E \to 0, \{W_{\bullet}^{i}\}) & \mapsto &
	\det E
\end{eqnarray*}
and $\bR(m)^{ss}(\ba)$ is the fiber of $[\cO] \in \Pic^{0}(C)$. Since $\omega_{\Pic^{0}(E)} \cong \cO$, $\omega_{\widetilde{\bR}(m)^{ss}(\ba)} \cong \omega_{\widetilde{\bR}(m)^{ss}(\ba)/\Pic^{0}(C)}$. Since $\det$ is a smooth fiber bundle, $\omega_{\bR(m)^{ss}(\ba)}$ is the restriction of $\omega_{\widetilde{\bR}(m)^{ss}(\ba)/\Pic^{0}(C)}$.
\end{remark}

\begin{proposition}\label{prop:canonicaldivisor}
Suppose that $\ba$ is a dominant weight. Then $\omega_{\bM_{X}(r,\cO, \ba)}^{-1}$ is the descent of the line bundle $\cL^{2r} \otimes \otimes_{i=1}^{n} F_{i, \lambda}$ on $\cM_{X}(r, \cO, \ba)$, where $\lambda = 2\sum_{j=1}^{r-1}\omega_{j}$. In particular,
\[
	\rH^{0}(\bM_{X}(r, \cO, \ba), \omega_{\bM_{X}(r, \cO, \ba)}^{-1})
	= \VV_{X, 2r, (\lambda, \lambda, \cdots, \lambda)}^{\dagger}.
\]
\end{proposition}

\begin{proof}
By \cite[Lemma 4.17]{NR93} and \cite[Satz 5]{Kn89}, $\omega_{\bM_{X}(r,\cO, \ba)}^{-1}$ is the descent of $\omega_{\bR(m)^{s}(\ba)}^{-1}$. Consider the following commutative diagram:
\[
	\xymatrix{\bR(m)^{s}(\ba) \ar[d] &
	\bS(m)^{s}(\ba) \ar[l]_{f}\ar[d]^{q}\\
	\bR(m)^{s}(\ba)/\SL_{\nu(m)} = \bR(m)^{s}(\ba)/\PGL_{\nu(m)}
	\ar@{=}[d]&
	[\bS(m)^{s}(\ba) /\GL_{\nu(m)}] \ar@{=}[d]\\
	\bM_{X}(r, \cO, \ba) & \cM_{X}(r, \cO, \ba) \ar[l]_{p}.}
\]
By construction, $\bS(m)^{s}(\ba) \cong \pi_{*}\det \cE \setminus \{0\}$. In particular, $f^{*}\det \cE|_{\bR(m)^{s}(\ba) \times \{y\}}$ has a non-vanishing global section and hence is trivial. Since $\det R\pi_{*}\cE$ and $\det \cQ_{p^{i}, j}$ are functorial, with abuse of notation,
\[
	f^{*}\omega_{\bR(m)^{s}(\ba)}^{-1} \cong
	(\det R\pi_{*}\cE)^{2r} \otimes
	\left(\bigotimes_{i=1}^{n}\bigotimes_{j=1}^{r-1}
	(\det\cQ_{p^i,j})^{2}\right).
\]
It descends to $\cL^{2r} \otimes \otimes_{i=1}^{n}F_{i, \lambda}$ because $q^{*}\cL \cong \det R\pi_{*}\cE$ and $q^{*}F_{i, \omega_{j}} = \det \cQ_{p^{i}, r-j}$.

Finally, by \textsf{Step 1} and \textsf{Step 2} of the proof of Proposition \ref{prop:coxringidentification},
\[
\begin{split}
	\rH^{0}(\bM_{X}(r, \cO, \ba), \omega_{\bM_{X}(r, \cO, \ba)}^{-1})
	&\cong
	\rH^{0}(\cM_{X}(r, \cO, \ba), \cL^{2r}\otimes \otimes_{i=1}^{n}
	F_{i, \lambda})\\
	&\cong		
	\rH^{0}(\cM_{X}(r, \cO), \cL^{2r}\otimes \otimes_{i=1}^{n}
	F_{i, \lambda}) \cong \VV_{X, 2r,(\lambda, \lambda, \cdots, \lambda)}^{\dagger}.
\end{split}
\]
\end{proof}

\subsection{Some lemmas}

In this section, we prove three lemmas that we will use for the proof of Theorem \ref{thm:mainthmintro}.

\begin{lemma}\label{lem:extension}
Let $\bq \supset \bp$ be any extended point configuration. Set $X = (C, \bp) \in \cM_{g, n}$ and $\widetilde{X} = (C, \bq)$. Let $\ba$ be a general parabolic weight for $X$. Then there is a parabolic weight $\ba'$ for $\widetilde{X}$ such that there are morphisms $\cM_{\widetilde{X}}(r, \cO, \ba') \to \cM_{X}(r, \cO, \ba)$ and $\bM_{\widetilde{X}}(r, \cO, \ba') \to \bM_X (r, \cO, \ba)$.
\end{lemma}

\begin{proof}
We may assume that $\bp = (p^{1}, p^{2}, \cdots, p^{n})$ and $\bsq = (p^{1}, p^{2}, \cdots, p^{n+m})$. Let $\ba'$ be a parabolic weight such that ${a'}_{\bullet}^{i} = a_{\bullet}^{i}$ for $i \le n$ and ${a'}_{j}^{i}$ are sufficiently small for $i > n$. There is a natural forgetful map
\begin{eqnarray*}
	\cM_{\widetilde{X}}(r, \cO, \ba') & \to & \cM_{X}(r, \cO, \ba)\\
	(E, \{W_{\bullet}^{i}\}) & \mapsto &
	(E, \{W_{\bullet}^{i}\}_{i \le n}).
\end{eqnarray*}
This map is regular, because small weights $(a_{\bullet}^{i})_{i > n}$ do not affect the inequalities for the stability. The morphism between coarse moduli spaces comes from the universal property of the coarse moduli space.
\end{proof}

The next lemma is a Mori theoretic interpretation of Theorem \ref{thm:Paulycoarse}.

\begin{lemma}\label{lem:projectivemodel}
Let $g \ge 1$ and let $\ba$ be a dominant weight. Let $D$ be a divisor on $\bM_{X}(r, \cO, \ba)$ such that $\cO(D)$ is the descent of $\cL^{\ell} \otimes \bigotimes_{i=1}^{n}F_{i, \lambda^{i}}$ such that $(\lambda^{i}, \theta) = \lambda_{1}^{i} < \ell$. Suppose that $\ba'$ is the parabolic weight such that $a_{j}^{i'} = \lambda_{j}^{i}/\ell$. Then $\bM_{X}(r, \cO, \ba)(D) \cong \bM_{X}(r, \cO, \ba')$.
\end{lemma}

\begin{proof}
As before, let $\vec{\lambda} = (\lambda^{1}, \lambda^{2}, \cdots, \lambda^{n})$.
\[
\begin{split}
	\rH^{0}(\bM_{X}(r, \cO, \ba), \cO(mD))
	&= \rH^{0}(\bM_{X}(r, \cO, \ba), \cL^{m\ell}\otimes
	\bigotimes_{i=1}^{n}F_{i, m\lambda^{i}})\\
	&= \rH^{0}(\cM_{X}(r, \cO, \ba), \cL^{m\ell}\otimes
	\bigotimes_{i=1}^{n}F_{i, m\lambda^{i}})\\
	&= \rH^{0}(\cM_{X}(r, \cO), \cL^{m\ell}\otimes
	\bigotimes_{i=1}^{n}F_{i, m\lambda^{i}})
	= \VV_{X, m\ell, m\vec{\lambda}}^{\dagger}.
\end{split}
\]
Thus
\[
\begin{split}
	\bM_{X}(r, \cO, \ba)(D) &= \proj \bigoplus_{m \ge 0}
	\rH^{0}(\bM_{X}(r, \cO, \ba), \cO(mD)) =
	\proj \bigoplus_{m \ge 0} \VV_{X, m\ell, m\vec{\lambda}}^{\dagger}\\
	&= \proj \bigoplus_{m \ge 0}
	\rH^{0}(\bM_{X}(r, \cO, \ba'), \cO(mD))
	= \bM_{X}(r, \cO, \ba')
\end{split}
\]
because $\cO(D)$ is ample on $\bM_{X}(r, \cO, \ba')$ (Theorem \ref{thm:Paulycoarse}).
\end{proof}

\begin{remark}
Note that $\ba'$ may be a degenerated parabolic weight (Remark \ref{rem:degeneratedpardata}) and that $\bM_{X}(r, \cO, \ba')$ is the moduli space of parabolic bundles with partial parabolic data.
\end{remark}

In general, a subalgebra of a finitely generated algebra may not be finitely generated. However, the next lemma shows the finite generation for a certain type of graded subalgebras.

\begin{lemma}\label{lem:invariantsubring}
Let $A$ be a free abelian group with finite rank. Let $R$ be a finitely generated $A$-graded $\CC$-algebra and let $\pi : R \to A$ be the grading map. Let $B \subset A$ be a subgroup and let $R_{B} := \{x \in R\;|\; \pi(x) \in B\}$. Then $R_{B}$ is also a finitely generated algebra.
\end{lemma}

\begin{proof}
It is straightforward to check that $R_{B}$ is a subalgebra. Let $G := \Hom(A, \CC^{*})$ and $H := \Hom(A/B, \CC^{*}) \le G$. Since $A$ is a finitely generated free abelian group, $A/B$ is a direct sum of a finite rank free abelian group and a finite abelian group. Thus $H$ is a direct sum of a torus and a finite abelian group. In particular, $H$ is reductive.

Consider the natural $G$-action on $R$:
\begin{eqnarray*}
	G \times R &\to & R\\
	(\phi, x) & \mapsto & \phi(\pi(x))x.
\end{eqnarray*}
Then there is a natural $H$-action on $R$ via the inclusion $H \to G$ and $R_{B} = R^{H}$. By Nagata's theorem (\cite[Theorem 3.3]{Dol03}), $R^{H}$ is finitely generated.
\end{proof}

\subsection{Proof of finite generation}\label{ssec:finitegeneration}

A projective variety $X$ is of \emph{Fano type} if there is a $\QQ$-divisor $\Delta$ such that $(X, \Delta)$ is log Fano.

\begin{theorem}\label{thm:fanotype}
For any general parabolic weight $\ba$, $\bM_{X}(r, \cO, \ba)$ is of Fano type.
\end{theorem}

\begin{proof}
First of all, assume that $\ba$ is dominant and sufficiently close to $\ba_{c}$. When $g = 1$, we further assume that $n > r$.

By Proposition \ref{prop:Picardgroupforcoarsemoduli}, $\bM := \bM_{X}(r, \cO, \ba)$ is of Picard number $n(r-1)+1$. By a wall-crossing, there is a contraction $c : \bM \to \bM_{X}(r, \cO, \ba_{c})$. This is a small contraction (unless it is an isomorphism) because the strictly semistable locus is of codimension at least two. $-K_{\bM} = -c^{*}K_{\bM_{X}(r, \cO, \ba_{c})}$ is nef and big since $-K_{\bM_{X}(r, \cO, \ba_{c})}$ is ample (Theorem \ref{thm:Paulycoarse}, Proposition \ref{prop:canonicaldivisor}) and $\bM$ and $\bM_{X}(r, \cO, \ba_{c})$ are birational. Thus $\bM$ is a smooth weak Fano variety. Thus $\bM$ is of Fano type (See the proof of \cite[Theorem 5.1]{MY17} for an argument.).

Suppose that $\ba$ is a general parabolic weight. When $g = 1$, we still assume that $n > r$. Pick a general parabolic weight $\bfb$ sufficiently close to $\ba_{c}$. By Lemma \ref{lem:projectivemodel}, $\bM_{X}(r, \cO, \ba) = \bM_{X}(r, \cO, \bfb)(D)$ for some $D$. Thus $\bM_{X}(r, \cO, \ba)$ is obtained from $\bM_{X}(r, \cO, \bfb)$ by taking several flips and blow-downs. By \cite[Theorem 1.1, Corollary 1.3]{GOST15}, $\bM_{X}(r, \cO, \ba)$ is also of Fano type.

When $g = 1$ and $n$ is small, by Lemma \ref{lem:extension}, $\bM_{X}(r, \cO, \ba)$ is an image of $\bM_{(C,\bsq)}(r, \cO, \ba')$ for some large $\bsq$. Thus it is of Fano type by \cite[Corollary 1.3]{GOST15}.
\end{proof}

The following is an immediate consequence of \cite[Corollary 1.3.2]{BCHM10}.

\begin{corollary}\label{cor:MDS}
For any general parabolic weight $\ba$, $\bM_{X}(r, \cO, \ba)$ is a Mori dream space.
\end{corollary}

Now we are ready to prove our main theorem for a smooth curve.

\begin{theorem}\label{thm:finitegeneration}
For any smooth $X = (C, \bp) \in \cM_{g, n}$, $\VV_{X}^{\dagger}$ is finitely generated.
\end{theorem}

\begin{proof}
The case of $g = 0$ is shown in \cite{MY17}. Assume that $g \ge 1$ and $n > r$. Let $\ba$ be a dominant weight. By Proposition \ref{prop:coxringidentification}, $\Cox(\cM_{X}(r, \cO))\cong\Cox(\bM_{X}(r, \cO, \ba))$. Since $\bM_{X}(r, \cO, \ba)$ is an MDS by Corollary \ref{cor:MDS}, $\VV_{X}^{\dagger} =  \Cox(\cM_{X}(r, \cO))$ is finitely generated.

When $g = 1$ and $n\le r$, take a sufficiently large point configuration $\bq \supset \bp$ and set $\widetilde{X} = (C, \bq)$. There is a forgetful morphism $\cM_{\widetilde{X}}(r, \cO) \to \cM_{X}(r, \cO)$ which forgets flags for $p^{i} \in \bq \setminus \bp$. Then $\Pic(\cM_{X}(r, \cO))$ is embedded into $\Pic(\cM_{\widetilde{X}}(r, \cO))$.
By propagation of vacua (Theorem \ref{thm:propagation}),
\[
	\VV_{X}^{\dagger} = \bigoplus_{\ell, \vec{\lambda}}
	\VV_{X, \ell, \vec{\lambda}}^{\dagger} \cong
	\bigoplus_{\ell, \vec{\lambda}}
	\VV_{\widetilde{X}, \ell,
	(\lambda^{1}, \lambda^{2}, \cdots, \lambda^{n},
	0, 0, \cdots, 0)}^{\dagger}.
\]
The last algebra is precisely $\left(\VV_{\widetilde{X}}^{\dagger}\right)_{\Pic(\cM_{X}(r, \cO))}$ (see Lemma \ref{lem:invariantsubring} for the notation). By Lemma \ref{lem:invariantsubring}, it is finitely generated.
\end{proof}

\section{Mori's program}\label{sec:Moriprogram}

Fix a smooth pointed curve $X = (C, \bp)$ of positive genus. Since $\bM_{X}(r, \cO, \ba)$ is an MDS by Corollary \ref{cor:MDS}, one may apply Mori's program at least theoretically, and classify all rational contractions of $\bM_{X}(r, \cO, \ba)$. In this section, we describe Mori's program for $\bM_{X}(r, \cO, \ba)$. For $g(C) = 0$ case, see \cite[Section 6]{MY17}.

By Theorem \ref{thm:Picstack}, every $\RR$-line bundle on $\cM_{X}(r, \cO)$, or equivalently, every $\RR$-line bundle on $\bM_{X}(r, \cO, \ba)$ for a dominant $\ba$ (Definition \ref{def:dominantweight}), can be written uniquely as
\[
	\cL^{\ell}\otimes \bigotimes_{i=1}^{n}\bigotimes_{j = 1}^{r-1}
	F_{i, \omega_{j}}^{d_{j}^{i}}
\]
for some $\ell, d_{j}^{i} \in \RR$. Note that $F_{i, \lambda} \cong \bigotimes_{j=1}^{r-1}F_{i, \omega_{j}}^{\lambda_{j} - \lambda_{j+1}}$.  Thus one may identify $\Pic(\cM_{X}(r, \cO))_{\RR}$ with $\RR^{n(r-1)+1}$ with coefficients $(\ell, d_{j}^{i})$.

\begin{definition}\label{def:coneE}
Let $E \subset \Pic(\cM_{X}(r, \cO))_{\RR} \cong \RR^{n(r-1)+1}$ be a convex polyhedral cone defined as the intersection of the following half spaces:
\begin{enumerate}
\item $\ell \ge 0$, $d_{j}^{i} \ge 0$;
\item $\sum_{j=1}^{r-1} d_{j}^{i} \le \ell$.
\end{enumerate}
Note that $E$ is strongly convex because this is a subcone of the first octant.
\end{definition}

The first step of Mori's program, which is the computation of the effective cone, is done by the next proposition.

\begin{proposition}\label{prop:effectivecone}
Let $\ba$ be a dominant parabolic weight. Then $\Eff(\bM_{X}(r, \cO, \ba)) = E$.
\end{proposition}

\begin{proof}
Let $D \in \mathrm{int} \; E$ be an integral divisor class and $(\ell, d_{j}^{i})$ is the coefficients of $D$. Set $\lambda^{i} := \sum_{j=1}^{r-1}d_{j}^{i}\omega_{j}$. Then $\cO(D) \cong \cL^{\ell}\otimes \bigotimes_{i=1}^{n}F_{i, \lambda^{i}}$, $(\lambda^{i}, \theta) = \lambda_{1}^{i} = \sum_{j=1}^{r-1}d_{j}^{i} < \ell$,  $\lambda_{j}^{i} > \lambda_{j+1}^{i}$. So by setting $b_{j}^{i} := \lambda_{j}^{i}/\ell$, we obtain a parabolic weight $\bfb$. By Lemma \ref{lem:projectivemodel}, $\bM_{X}(r, \cO, \ba)(D) = \bM_{X}(r, \cO, \bfb)$. Since $\bM_{X}(r, \cO, \bfb) \ne \emptyset$ (Corollary \ref{cor:nonempty}), $\rH^{0}(\cO(mD)) \ne 0$ for some $m > 0$. Moreover, since $\bM_{X}(r, \cO, \bfb)$ is birational to $\bM_{X}(r, \cO, \ba)$, $D$ is big and $D \in \mathrm{int}\; \Eff(\bM_{X}(r, \cO, \ba))$. Thus $E$ is in the closure of $\Eff(\bM_{X}(r, \cO, \ba))$. The latter cone is closed since $\bM_{X}(r, \cO, \ba)$ is an MDS. Hence, $E \subset \Eff(\bM_{X}(r, \cO, \ba))$.

If $D$ is not in $E$, then $D$ violates one of the inequalities in Definition \ref{def:coneE}. Then the associated conformal block $\VV_{X, \ell, \vec{\lambda}}^{\dagger}$ is trivial. Thus the linear system $|D|$ on $\bM_{X}(r, \cO, \ba)$ is empty. Thus $D \notin \Eff(\bM_{X}(r, \cO, \ba))$.
\end{proof}

Lemma \ref{lem:projectivemodel} says that for any big divisor $D$ of $\bM_{X}(r, \cO, \ba)$, $\bM_{X}(r, \cO, \ba)(D) = \bM_{X}(r, \cO, \bfb)$ for some parabolic weight $\bfb$. Thus \emph{all} birational models that we may obtain while applying Mori's program are moduli spaces of parabolic bundles for some $\bfb$, and any birational rational contractions of $\bM_{X}(r, \cO, \ba)$ are described in terms of wall-crossings of moduli spaces of parabolic bundles (Section \ref{ssec:wallcrossing}). In particular, if a birational model is also smooth, the rational contraction is a composition of smooth blow-ups and blow-downs. For a non big divisor, rational contractions are described by Lemma \ref{lem:extension} and the following lemma.

\begin{lemma}
Let $\ba$ be a dominant parabolic weight. Suppose that $D$ is a divisor on $\bM_{X}(r, \cO, \ba)$ such that $\cO(D)$ is isomorphic to the descent of $\cL^{\ell} \otimes \bigotimes_{i=1}^{n}F_{i, \lambda^{i}}$, and $(\lambda^{k}, \theta) = \lambda_{1}^{k} = \ell$ and $(\lambda^{i}, \theta) < \ell$ for $i \ne k$. Then
\[
	\bM_{X}(r, \cO, \ba)(D) \cong \bM_{X}(r, \cO(-p^{k}), \bfb)
\]
where $\bfb$ is a partial parabolic weight such that $b^{i} = \frac{1}{\ell}(\lambda_{1}^{i}, \lambda_{2}^{i}, \cdots, \lambda_{r-1}^{i})$ for $i \ne k$ and $b^{k} = \frac{1}{\ell}(\lambda_{1}^{k}-\lambda_{r-1}^{k}, \lambda_{2}^{k} - \lambda_{r-1}^{k}, \cdots, \lambda_{r-2}^{k} - \lambda_{r-1}^{k})$ (the last flag is of type $(2, 3, \cdots, r-1)$).
\end{lemma}

\begin{proof}
The proof is identical to that of \cite[Proposition 6.7]{MY17}. Here we give an outline of the proof. We may assume that $k = n$.

First of all, suppose that $\ba$ is sufficiently close to $\frac{1}{\ell}(\lambda^{i})$. For any $\cE = (E, \{W_{\bullet}^{i}\}) \in \cM_{X}(r, \cO, \ba)$, let $\cE' = (E', \{W_{\bullet}^{'i}\})$ be a parabolic bundle obtained as follows. Let $E'$ be the kernel of the quotient map $E \to E|_{p^{n}} \to E|_{p^{n}}/W_{r-1}^{n} \to 0$. Let $W_{j}^{'i} = \iota_{i}^{-1}W_{j}^{i}$, where $\iota_{i} : E'|_{p^{i}} \to E|_{p^{i}}$. Note that $\det E' \cong \cO(-p^{n})$.

By the computation in the proof of \cite[Proposition 6.7]{MY17}, one may check that $\cE'$ is stable with respect to $\bfb$. Thus we have a functorial morphism $\bM_{X}(r, \cO, \ba) \to \bM_{X}(r, \cO(-p^{n}), \bfb)$. This is a $\PP^{1}$-fibration and of relative Picard number one. Thus $\bM_{X}(r, \cO(-p^{n}), \bfb)$ is a projective model of $\bM_{X}(r, \cO, \ba)$ associated to a nef but not a big divisor lying on a facet of $\Eff(\bM_{X}(r, \cO, \ba))$. Since the projective models for the facets $d_{j}^{i} = 0$ are given by Lemma \ref{lem:projectivemodel}, $\bM_{X}(r, \cO(-p^{n}), \bfb)$ is associated to one of the facets of the form $\sum_{j=1}^{r-1}d_{j}^{i} = \ell$ for some $i$. It corresponds to the facet for $i = n$ because for the bundles parametrized by a fiber $f$ of $\bM_{X}(r, \cO, \ba) \to \bM_{X}(r, \cO(-p^{n}), \bfb)$, the flags $\{W_{\bullet}^{i}\}$ for $i \ne n$ are constant and the intersection number of $f$ with the theta divisor does not change while we vary $a^{i}$ and $\ba$ approaches another facet $\sum_{j=1}^{r-1}d_{j}^{i} = \ell$ for $i \ne n$.

For the general case, take an ample $\QQ$-divisor $A$ on $\bM_{X}(r, \cO, \ba)$ and run directed MMP for $D(t) = (1-t)A + tD$, $0 \le t \le 1$. We may assume that while running MMP, every wall-crossing is a simple one by choosing a general $A$. For a small rational $\epsilon > 0$, the rational contraction $\bM_{X}(r, \cO, \ba) \dashrightarrow \bM_{X}(r, \cO, \ba)(D)$ is a composition of $\bM_{X}(r, \cO, \ba) \dashrightarrow \bM_{X}(r, \cO, \ba)(D(1-\epsilon)) \dashrightarrow \bM_{X}(r, \cO, \ba)(D(1)) = \bM_{X}(r, \cO, \ba)(D)$. Since $D(1-\epsilon)$ is big and $\bM_{X}(r, \cO, \ba)(D(1-\epsilon)) \cong \bM_{X}(r, \cO, \ba')$ for some $\ba'$ by Lemma \ref{lem:projectivemodel}. Because $\ba'$ is sufficiently close to $\frac{1}{\ell}(\lambda^{i})$, there is a regular morphism $\bM_{X}(r, \cO, \ba') \to \bM_{X}(r, \cO(-p^{k}), \bfb)$ as previously. Thus there is a rational contraction $\bM_{X}(r, \cO, \ba) \dashrightarrow \bM_{X}(r, \cO(-p^{k}), \bfb)$ with positive dimensional fibers. By the same argument we may conclude that this is a projective model associated to the facet $\lambda_{1}^{k} = \ell$.
\end{proof}

\section{Singular curve case}\label{sec:singularcurve}

The finite generation of $\VV_{X}^{\dagger}$ for smooth pointed curves, which was shown in \cite{MY17}, and Theorem \ref{thm:finitegeneration} implies the finite generation of $\VV_{X}^{\dagger}$ for arbitrary stable curves. As a corollary, we construct a flat family of irreducible normal projective varieties over $\overline{\cM}_{g, n}$, which extends the relative moduli space of parabolic vector bundles. It is a generalization of \cite[Theorem 1.2]{BG19}.

\begin{theorem}\label{thm:finitegenerationsingular}
Let $X = (C, \bp) \in \overline{\cM}_{g, n}$. Then the algebra $\VV_{X}^{\dagger}$ is a finitely generated integrally closed domain.
\end{theorem}

It is enough to prove for $X \in \overline{\cM}_{g, n}\setminus \cM_{g, n}$.

Let $\nu : \widetilde{C} = \sqcup_{j \in J} C_{j} \to C$ be the normalization map. For each singular point $x \in C$, we denote two points in $\nu^{-1}(x)$ by $p_{x}$ and $q_{x}$. A special point on $\widetilde{C}$ is either the inverse image of a marked point or that of a singular point. By abuse of notation, we denote $\nu^{-1}(p^{i})$ by $p^{i}$. For any connected component $C_{j}$ of $\widetilde{C}$, we denote the set of special points on $C_{j}$ by $\bq_{j}$. $\bq_{j}^{sing} \subset \bq_{j}$ is the set of the inverse images of singular points. Let $X_{j} = (C_{j}, \bq_{j})$. Finally, let $\bq = \sqcup_{j \in J}\bq_{j}$, $\bq^{sing} = \sqcup_{j \in J}\bq_{j}^{sing}$, and let $\widetilde{X} = (\widetilde{C}, \bq)$.

Let $S_{\ell}$ be the set of sequences $\vec{\mu} = (\mu^{y})_{y \in \bq^{sing}}$, where each $\mu^{y}$ is a dominant integral weight satisfying $(\mu^{y}, \theta) \le \ell$ and $\mu^{q_{x}} = \mu^{p_{x} *}$ for any singular point $x \in C$. For any $j \in J$, where $J$ is the set of irreducible components of $C$, any $\vec{\lambda} = (\lambda^{1}, \lambda^{2}, \cdots, \lambda^{n})$, and $\vec{\mu} \in S_{\ell}$, let $(\vec{\lambda} \star \vec{\mu})_{j}$ be the assignment of dominant integral weights for $\bq_{j}$ defined as:
\begin{enumerate}
\item If $y = p^{i}\in \bq_{j}\setminus \bq_{j}^{sing}$, assign $\lambda^{i}$;
\item If $y \in \bq_{j}^{sing}$, assign $\mu^{y}$.
\end{enumerate}

Consider the algebra
\[
	\bigotimes_{j \in J}\VV_{X_{j}}^{\dagger}
	= \bigoplus \bigotimes_{j \in J}
	\VV_{X_{j}, \ell_{j}, \vec{\lambda}_{j}}^{\dagger}.
\]

By applying factorization repeatedly, we have
\[
	\VV_{X, \ell, \vec{\lambda}}^{\dagger} \cong
	\bigoplus_{\vec{\mu} \in S_{\ell}}
	\bigotimes_{j \in J}\VV_{X_{j}, \ell,
	(\vec{\lambda}\star \vec{\mu})_{j}}^{\dagger}.
\]
Thus there is a natural injective map
\begin{equation}\label{eqn:algebrafactorization}
	\VV_{X}^{\dagger} \to \bigotimes_{j \in J}\VV_{X_{j}}^{\dagger}.
\end{equation}

\begin{proposition}[(\protect{\cite[Proposition 3.1]{Man18}})]\label{prop:factorizationalgebra}
The map \eqref{eqn:algebrafactorization} is an algebra homomorphism.
\end{proposition}

\begin{proof}
For notational simplicity, we give the proof for when $X = (C, \bp)$ has two irreducible components $X_{1} = (C_{1}, \bq_{1})$ and $X_{2} = (C_{2}, \bq_{2})$ and there is a single node. The general case is obtained by repeating the same argument. To show that $\VV_{X}^{\dagger} \to \VV_{X_{1}}^{\dagger}\otimes \VV_{X_{2}}^{\dagger}$ is an algebra homomorphism, we prove that there is a commutative diagram
\[
\xymatrix{\VV_{X, \ell, \vec{\lambda}}^{\dagger} \otimes \VV_{X, m, \vec{\mu}}^{\dagger} \ar[r]^(.24){\sim} \ar[d]_{\otimes} & \left(\bigoplus\VV_{X_{1}, \ell, \{\vec{\lambda}_{1}, \rho\}}^{\dagger}\hspace{-3pt}\otimes\hspace{-3pt} \VV_{X_{2}, \ell, \{\vec{\lambda}_{2}, \rho^{*}\}}^{\dagger}\right)\hspace{-3pt}\otimes\hspace{-3pt} \left(\bigoplus\VV_{X_{1}, m, \{\vec{\mu}_{1}, \tau\}}^{\dagger}\hspace{-3pt}\otimes\hspace{-3pt} \VV_{X_{2}, m, \{\vec{\mu}_{2}, \tau^{*}\}}^{\dagger}\right)\ar[d]^{\otimes}\\
\VV_{X, \ell+m, \vec{\lambda}+\vec{\mu}}^{\dagger} \ar[r]^(.3){\sim} &\left(\bigoplus \VV_{X_{1}, \ell+m, \{\vec{\lambda}_{1}+\vec{\mu}_{1}, \rho+\tau\}}^{\dagger}\otimes \VV_{X_{2}, \ell+m, \{\vec{\lambda}_{2}+\vec{\mu}_{2}, \rho^{*}+\tau^{*}\}}^{\dagger}\right).}
\]
Note that two vertical maps are induced from the restriction of the domain ($H_{\ell+m, \lambda + \mu} \to H_{\ell, \lambda} \otimes H_{m, \mu}$, see Section \ref{ssec:conformalblocks}). The symbol $\sim$ over horizontal arrows denotes the assignment $\varphi_{\rho} \mapsto \tilde{\varphi}_{\rho}$ which we will explain later.

For each integral partition $\lambda$, there is a one-dimensional trivial subrepresentation of $V_{\lambda}\otimes V_{\lambda^{*}}$. Let $\mathbf{1}_{\lambda, \lambda^{*}}$ be a nonzero vector of $V_{\lambda}\otimes V_{\lambda^{*}} \cong \Hom_{\CC}(V_{\lambda}, V_{\lambda})$, which corresponds to the identity matrix. The vectors $\mathbf{1}_{\lambda, \lambda^{*}}$ are compatible in the sense that the image of $\mathbf{1}_{\lambda, \lambda^{*}} \otimes \mathbf{1}_{\mu, \mu^{*}} \in V_{\lambda}\otimes V_{\lambda^{*}} \otimes V_{\mu}\otimes V_{\mu^{*}}$ in $V_{\lambda+\mu} \otimes V_{\lambda^{*}+\mu^{*}}$ is $\mathbf{1}_{\lambda+\mu, \lambda^{*}+\mu^{*}}$. Now the factorization map is defined as follows. For each $\varphi \in \VV_{X, \ell, \vec{\lambda}}^{\dagger}$, there is a unique linear combination $\varphi = \sum_{\rho, (\rho, \theta) \le \ell}\varphi_{\rho}$, and for each $\rho$, there is a unique $\tilde{\varphi}_{\rho} \in \VV_{X_{1}, \ell, \{\vec{\lambda}_{1}, \rho\}}^{\dagger}\otimes \VV_{X_{2}, \ell, \{\vec{\lambda}_{2}, \rho^{*}\}}^{\dagger}$ such that $\tilde{\varphi}_{\rho}(v \otimes \textbf{1}_{\rho, \rho^{*}}) = \varphi_{\rho}(v)$ (\cite[Section 3.3.2]{Uen08}).

For any basis $\varphi_{\rho} \in \VV_{X, \ell, \vec{\lambda}}^{\dagger}$ and $\psi_{\tau} \in \VV_{X, m, \vec{\mu}}^{\dagger}$ associated to $\tilde{\varphi}_{\rho}$ and $\tilde{\psi}_{\tau}$ respectively and for any $v \otimes w \in \VV_{X, \ell+m, \vec{\lambda}+\vec{\mu}}$,
\[
\begin{split}
(\tilde{\varphi}_{\rho} \otimes \tilde{\psi}_{\tau})(v \otimes w \otimes \textbf{1}_{\rho+\tau, \rho^{*}+\tau^{*}}) &= (\tilde{\varphi}_{\rho} \otimes \tilde{\psi}_{\tau})(v \otimes w \otimes \textbf{1}_{\rho, \rho^{*}}\otimes \textbf{1}_{\tau, \tau^{*}})\\
& = \tilde{\varphi}_{\rho}(v \otimes \textbf{1}_{\rho, \rho^{*}})\tilde{\psi}_{\tau}(w \otimes \textbf{1}_{\tau, \tau^{*}}) = \varphi_{\rho}(v)\psi_{\tau}(w)\\
& = (\varphi_{\rho}\otimes \psi_{\tau})(v\otimes w) = (\widetilde{\varphi_{\rho}\otimes \psi_{\tau}})(v\otimes w \otimes \textbf{1}_{\rho+\tau, \rho^{*}+\tau^{*}}).
\end{split}
\]
Thus we obtain $\tilde{\varphi}_{\rho} \otimes \tilde{\psi}_{\tau} = \widetilde{\varphi_{\rho}\otimes \psi_{\tau}}$, which shows that the two compositions are same.
\end{proof}

Now we can complete the proof. We retain the same notation.

\begin{proof}[of Theorem \ref{thm:finitegenerationsingular}]
Consider $\bigotimes_{j \in J}\VV_{X_{j}}^{\dagger}$.This is a $\prod_{j \in J}\Pic(\cM_{X_{j}}(r, \cO))$-graded $\CC$-algebra. By Theorem \ref{thm:finitegeneration}, this is finitely generated. Each $\VV_{X_{j}}^{\dagger}$ is either the Cox ring of a normal projective variety or its torus invariant subring. Thus it is integrally closed (\cite[Corollary 1.2]{EKW04} and \cite[Proposition 3.1]{Dol03}). The tensor product over the base field $\CC$ of integrally closed domains is also integrally closed.

By Proposition \ref{prop:factorizationalgebra}, $\VV_{X}^{\dagger}$ is a subalgebra of $\bigotimes_{j \in J}\VV_{X_{j}}^{\dagger}$. More precisely, let
\[
	K := \{(\cL^{\ell_{j}}\otimes \otimes_{y \in \bq_{j}}F_{y, \lambda^{y}})
	\;|\; \ell_{j} = \ell_{k} \;\forall j, k \in J,
	\lambda^{q_{x}} = \lambda^{p_{x} *}\;
	\forall p_{x}, q_{x} \in \bq^{sing}\}.
\]
Then $K$ is a saturated subgroup of $\prod_{j \in J}\Pic(\cM_{X_{j}}(r, \cO))$ defined by finitely many linear equations. Now $\VV_{X}^{\dagger} \cong \left(\bigotimes_{j \in J}\VV_{X_{j}}^{\dagger}\right)_{K}$. By Lemma \ref{lem:invariantsubring}, this is finitely generated and integrally closed (\cite[Proposition 3.1]{Dol03}).
\end{proof}

\begin{proof}[of Theorem \ref{thm:applicationintro}]
The algebra of conformal blocks $\VV_{X}^{\dagger}$ forms a flat sheaf $\VV^{\dagger}$ of finitely generated algebras over $\overline{\cM}_{g, n}$. Pick $X = (C, \bp) \in \cM_{g, n}$ and let $A := \Pic(\cM_{X}(r, \cO))$. Then $\VV^{\dagger}$ is a sheaf of $A$-graded algebra. For any $\ba$, we may find $\ell \in \ZZ_{\ge 0}$ and partitions $\lambda^{i}$ such that $a_{j}^{i} = \lambda_{j}^{i}/\ell$. By putting linear equations $a_{j}^{i}\ell = \lambda_{j}^{i}$, we may define a subgroup $B \le A$. Let $\VV_{\ba}^{\dagger} := \left(\VV^{\dagger}\right)_{B}$ be the sheaf of $B \cong \ZZ$-graded algebras $\VV_{\ba, X}^{\dagger} := \left(\VV_{X}^{\dagger}\right)_{B}$. By Lemma \ref{lem:invariantsubring}, each fiber is a finitely generated domain. Since $\VV^{\dagger}_{X}$ is integrally closed, $\left(\VV_{X}^{\dagger}\right)_{B}$ is also integrally closed. By taking
\[
	\cY := \proj \VV_{\ba}^{\dagger},
\]
we obtain the desired result.
\end{proof}

\begin{remark}
Note that if $n = 0$, then we recover \cite[Theorem 1.2]{BG19}.
\end{remark}

It would be a very interesting problem to describe special fibers as moduli spaces of some natural objects. We do not know the precise description yet. See \cite[Section 11.2]{BG19} for some discussion.

\begin{remark}
In \cite{Man18}, Manon describes a flat degeneration of $\VV_{X}^{\dagger}$. His degeneration is $\VV_{X'}^{\dagger}$ where $X' \in \overline{\cM}_{g, n}$ is a maximally degenerated curve whose normalization is a union of three pointed rational curves.
\end{remark}


\end{document}